\documentclass[12pt,a4paper]{amsart}
\usepackage{amssymb,amscd,amsthm,amsmath}
\usepackage{amsfonts}
\usepackage[top=35mm, bottom=35mm, left=30mm, right=30mm]{geometry}
\usepackage[colorlinks=true,citecolor=blue]{hyperref}
\usepackage{mathptmx}
\usepackage{eucal}
\usepackage{graphicx}
\usepackage{mathrsfs}
\usepackage{xcolor}
\usepackage[pagewise]{lineno}\nolinenumbers

\newtheorem{thm}{Theorem}[section]
\newtheorem{cor}[thm]{Corollary}
\newtheorem{lem}[thm]{Lemma}
\newtheorem{prop}[thm]{Proposition}

\theoremstyle{definition}
\newtheorem{df}[thm]{Definition}
\newtheorem{ex}[thm]{Example}
\newtheorem{rem}[thm]{Remark}
\numberwithin{equation}{section}
\def\UU{\mathcal{U}}
\def\diam{\text{\rm diam}}
\def\Leb{\text{\rm Leb}}
\def\logf{\log\frac{1}{\epsilon}}

\makeatletter

\newcommand{\Rmnum}[1]{\expandafter\@slowromancap\romannumeral #1@}
\makeatother

\begin{document}

\title{Upper metric mean dimensions with potential of  $\epsilon$-stable sets}

\author{Rui Yang$^{1}$, Ercai Chen$^{2}$ and Xiaoyao Zhou$^{2,*}$
}
\address{1. School of Mathematics, Northwest University, Xi'an, Shaanxi 710127, P.R. China}
\address
{2. School of Mathematical Sciences and Institute of Mathematics, Ministry of Education Key Laboratory of NSLSCS, Nanjing Normal University, Nanjing, 210023, Jiangsu, P.R. China}

\email{zkyangrui2015@163.com}
\email{ecchen@njnu.edu.cn}
\email{zhouxiaoyaodeyouxian@126.com}

\renewcommand{\thefootnote}{}

\footnotetext{* Corresponding author.}


 
\renewcommand{\thefootnote}{\arabic{footnote}}
\setcounter{footnote}{0}

\begin{abstract}
	
For infinite-entropy systems, we investigate chaotic behavior occurring in stable sets and use upper metric mean dimensions to characterize  the ``size" of $\epsilon$-stable sets and stable sets. The main results of this paper are as  follows:

(1) We establish precise relations between the upper metric mean dimensions (with  potential) of
$\epsilon$-stable sets   and  those of the whole phase space. As an application, we show that the corresponding upper metric mean dimensions defined by topological conditional entropy, preimage neighborhood entropy, and topological entropy coincide.

(2) We prove that for every ergodic measure  with positive upper mean Katok entropy, weakly mixing sets occur in the closures of stable sets for almost every point and have positive  upper capacity metric mean dimension.  Furthermore, the upper metric mean dimension of the whole phase space equals the divergence rate, as the scale tends to zero, of the supremum of the $\epsilon$-topological entropies of its stable sets.
\end{abstract}


\maketitle
\pagestyle{plain}  
\tableofcontents 

\section{Introduction} 
By a triplet $(X, d,T)$, we mean a topological dynamical system (TDS for short), where $(X,d)$ is a compact  metric space, and  $T: X\rightarrow X$ is a continuous map.  Various authors have introduced entropy-like  quantities \cite{akm65, bow71, b73} and pressure-like  quantities \cite{rue73, w82, pp84} to quantify the topological complexity of dynamical systems. However, these topological invariants fail to capture additional dynamical information for infinite-entropy systems.
To characterize the topological complexity of infinite-entropy  systems, Lindenstrauss and Weiss \cite{lw00} introduced the concept of metric mean dimension by dynamicalizing the Minkowski dimension in fractal geometry. This quantity measures the  rate at which  $\epsilon$-topological entropy  diverges to $\infty$ as $\epsilon \to 0$. To obtain an analogue of topological pressure in the setting of metric mean dimension, Tsukamoto \cite{t20} introduced metric mean dimension with potential (analogous to the definition of topological pressure), which was further studied  by  several authors  \cite{cls21, ycz22} from the perspective of dimension theory.

\emph{The goal of this paper is to use upper metric mean dimension (with potential) to study chaotic phenomena in stable sets, and to compute the upper metric mean dimensions of  $\epsilon$-stable sets and stable sets.}  Let $x\in X$ and $n\in \mathbb{N}$ and define $$W^s_{\epsilon, [0,n-1]}(x,T):=\{y\in X:d(T^jx,T^jy)\leq\epsilon, \forall~ 0\leq j <n\}.$$  Recall that for any non-empty subset $Z\subset X$,  the \emph{upper capacity topological entropy of $Z$} is  defined by
$$h_{top}^{UC}(T, Z)=\lim\limits_{\epsilon \to 0} h_{top}^{UC}(T,Z,d,\epsilon)=\lim\limits_{\epsilon \to 0}\limsup_{n \to \infty}\frac{1}{n} \log r_n(T,Z,d,\epsilon),$$
where  $r_n(T,Z,d,\epsilon)$ denotes   the  minimum number of    $\epsilon$-stable blocks $W_{\epsilon, [0,n-1]}^s(x,T)$   required to cover  $Z$. Generally speaking, the precise values of the upper capacity topological entropy can be calculated only for a few dynamical systems. In 1972,   Bowen \cite{bow72} established an inequality that bounds the topological entropy  of $X$ from above by the sum of the $\epsilon$-topological entropy  of $X$ and the upper capacity topological entropy of  $\epsilon$-stable sets, i.e.,  for every  $\epsilon>0$, 
\begin{align*}
	h_{top}(T,X)\leq   h_{top}(T,X,d,\epsilon)+\sup_{x\in X} h_{top}(T, W_\epsilon^s(x,T)),
\end{align*}
where  $W_\epsilon^s(x,T)=\{y\in X:d(T^jx,T^jy)\leq\epsilon, \text{for all }j=0,1,...\}$  denotes  the $\epsilon$-stable  set of $x$ (consisting of points  $y\in X$  whose orbits $\epsilon$-shadow the orbit of $x$ for all $n \geq 0$). Analogous inequalities also hold for the Bowen and packing topological entropies of $Z$ \cite{wa22}. Observe that,  for finite  entropy systems (e.g., $h$-expansive systems),  the second term $\sup_{x\in X} h_{top}(T, W_\epsilon^s(x,T))$  in  the above inequality is  basically ``negligible" for sufficiently small $\epsilon>0$. However, for infinite-entropy systems,  this term equals $\infty$ for every $\epsilon>0$. In 2003,  Fiebig, Fiebig and Nitecki \cite{ffn03} defined the  quantity  $$h_s(T,x,\epsilon)=\lim\limits_{\delta \to 0} \limsup_{n \to \infty}\frac{1}{n}\log r_n(T, T^{-n}W_\epsilon^s(x,T),d,\delta),$$  which describes  the dispersion of preimages of  the  $\epsilon$-stable set,   and  showed  that  if $X$ has finite (Lebesgue) covering dimension,  then 
\begin{align*}
	h_{top}(T,X)=\sup_{x\in X} h_s(T,x,\epsilon)
\end{align*}
for every $\epsilon >0$.  
Using the ``excellent partition lemma" (a key tool developed in \cite[Lemma 4]{bhr02}) and ergodic theory, Huang \cite{h08} proved that the finite-dimensionality assumption is unnecessary for every surjective non-invertible   dynamical system.  

\emph{Question 1: For infinite-entropy systems, can  one use metric mean dimension to obtain analogous  refined  entropy formulae for the aforementioned three types of  $\epsilon$-stable sets?}  

We note that developing an analogue of the ``excellent partition lemma" for infinite-entropy systems is non-trivial. Without invoking ergodic theory, we address Question 1 by topological methods and obtain the following main results. More precisely,  for every  $\epsilon >0$ and $x\in X$, one  can define the following objects:\\
(i) the  generalized   $\epsilon$-stable set $W^s_{\epsilon,  \mathcal{S}}(x,T)$ of $x$  by replacing $\mathbb{N}$ with  an  infinite subset $\mathcal{S} \subset \mathbb{N}$;\\
(ii) the blocks of   $\epsilon$-stable sets  $W^s_{\epsilon,  \mathcal{S}^{*}}(x,T)$  of $x$ by replacing   $[0,n-1]$  with  a   finite subset $\mathcal{S}^{*} \subset \mathbb{N}$. 

These definitions relax the shadowing requirement by allowing it to be imposed only at a prescribed set of times along the orbit of $x$.
 Let $f: X\rightarrow \mathbb{R}$ be a continuous potential. We denote by  $\overline{\rm mdim}_M^{UC}(T,f,Z,d)$, $\overline{\rm mdim}_M^B(T,f,Z,d)$, and  $\overline{\rm mdim}_M^P(T,f,Z,d)$  the upper capacity metric mean dimension with potential $f$  on $Z$, the Bowen and packing upper metric mean dimensions with potential $f$  on $Z$, respectively.

\begin{thm}\label{thm 1.1}
Let $(X,d,T)$ be a TDS  and $f\in C(X, \mathbb{R})$.  Then  for every $\epsilon >0 $ and every  finite subset  $\mathcal{S}^{*} \subset \mathbb{N}$,
\begin{align*}
\overline{\rm mdim}_M(T,f,X,d)&=\max_{x\in X}\overline{\rm mdim}_M^B(T,f,W_{\epsilon,   \mathcal{S}^{*}}^s(x,T),d)\\
&=\max_{x\in X}\overline{\rm mdim}_M^P(T,f,W_{\epsilon,   \mathcal{S}^{*}}^s(x,T),d)\\
&=\max_{x\in X}\overline{\rm mdim}_M^{UC}(T,f,W_{\epsilon,   \mathcal{S}^{*}}^s(x,T),d).
\end{align*} 
\end{thm}

\begin{thm}\label{thm 1.2}
Let $(X,d,T)$ be a TDS  and $f\in C(X, \mathbb{R})$.  Then  for every  infinite subset $\mathcal{S} \subset \mathbb{N}$ and   every  $\epsilon >0$, 
\begin{align*}
\overline{\rm mdim}_M(T,f,X,d)&=\limsup_{\delta\to 0}\frac{1}{\log \frac{1}{\delta}}\sup_{x\in X}M_{\delta}(T,f,W_{\epsilon, \mathcal{S}}^s(x,T),d)\\
&=
\limsup_{\delta\to 0}\frac{1}{\log \frac{1}{\delta}}\sup_{x\in X} \mathcal{P}_{\delta}(T,f,W_{\epsilon, \mathcal{S}}^s(x,T),d)\\
&=
\limsup_{\delta\to 0}\frac{1}{\log \frac{1}{\delta}}\sup_{x\in X} P(T,f, W_{\epsilon, \mathcal{S}}^s(x,T),d,\delta).
\end{align*}   
\end{thm}

For simplicity, we  define the  \emph{exponential growth rate} of a   sequence $\{p_n\}_{n\geq 1}$ of positive real numbers as
$$GR(\{p_n\}_{n\geq 1})=\limsup_{n \to \infty}\frac{\log p_n}{n}.$$  
\begin{thm}\label{thm 1.3}
Let $(X,d,T)$ be a TDS  and $f\in C(X, \mathbb{R})$.  Then  for every $\epsilon >0 $, 
\begin{align*}
\overline{\rm mdim}_M(T,f,X,d)=\limsup_{\delta\to 0}\frac{1}{\log \frac{1}{\delta}}GR(\{\sup_{x\in X} P_n(T,f, T^{-n}W_\epsilon^s(x,T),d,\delta)\}_{ n\geq 1}).
\end{align*}   
\end{thm}

It is well known that for positive-entropy systems, chaotic behavior (e.g., Li-Yorke scrambled sets \cite{bhr02, k98}, perfect 
$*$-chaotic sets \cite{s03}) can arise within  subsets of stable sets. In \cite{h08}, Huang established the following chaotic property for (invertible) positive-entropy dynamical systems: there exists a set that is sufficiently ``large" in a measure-theoretic sense such that the closure of the stable or unstable set of every point in this set contains a weakly mixing set. This naturally leads to the following question:

\emph{Question 2:  For infinite-entropy systems, can one  use upper metric mean dimension to characterize the  ``size" of  weakly mixing sets  appearing in closures of  stable sets?}

Two measure-theoretic counterparts of  metric mean dimension  for an invariant measure $\mu$ are the \emph{upper mean Rényi information dimension} \cite{gs20} and the \emph{upper mean Katok's entropy} \cite{ycz25}, i.e.,
\begin{align*}
\overline{\rm MRID}_{\mu}(T,X,d):&=\limsup_{\epsilon\to 0} \frac{1}{\logf}\inf_{\mathrm{diam}(\alpha)\leq \epsilon}h_{\mu}(T,\alpha)\\
\overline{\rm mdim}_{\mu}^K(T,X,d):&=\limsup_{\epsilon\to 0}\frac{\overline{h}_\mu^{K}(T,\epsilon)}{\logf}.
\end{align*}
where $\inf_{\mathrm{diam}(\alpha)\leq \epsilon}h_{\mu}(T,\alpha)$ denotes the  Kolmogorov-Sinai $\epsilon$-entropy of $\mu$ with $\alpha$ being taken over all finite Borel partitions of $X$ with  diameter at most $ \epsilon$,  and $\overline{h}_\mu^{K}(T,\epsilon)$ denotes  Katok's $\epsilon$-entropy of $\mu$.
These  two  notions allow us to  connect the theory of chaoticity  with  mean dimension theory.

\begin{thm} \label{thm 1.4}
Let $(X,d,T)$ be a TDS,  and let  $\mu$ be a $T$-invariant ergodic   Borel probability measure with $\overline{\rm mdim}_{\mu}^K(T,X,d)>0$. Then for $\mu$-a.e. $x \in X$,  there exists   a closed  subset $E(x) \subset \overline{W^s(x,T)}$ such that  
\begin{enumerate}
\item [(i)]$E(x)$ is a weakly mixing set for $T$;
\item  [(ii)]$\overline{\rm {mdim}}_M^{UC}(T,E(x),d)\geq \overline{\rm mdim}_{\mu}^K(T,X,d).$
\end{enumerate}
\end{thm}

Additionally, the  divergence rate of  the $\epsilon$-topological entropy of stable sets is closely related to the upper metric mean dimension of the phase space.

\begin{thm} \label{thm 1.5}
Let  $(X,d,T)$ be a TDS. Then 
$$\overline{\rm {mdim}}_M(T,X,d)=\limsup_{\delta \to 0}\frac{1}{\log\frac{1}{\delta}}\sup_{x\in X}h_{top}(T,W^s(x,T),d,\delta).$$
\end{thm}

The rest of this paper is organized as follows: In Section 2, we recall the definitions of three types of upper metric mean dimension with potential on subsets and state their basic properties. In Section 3, we compute the upper metric mean dimension of 
$\epsilon$-stable sets and prove  Theorems \ref{thm 1.1}-\ref{thm 1.3}. In Section 4,  we study  chaoticity  in infinite-entropy systems and prove  Theorems   \ref{thm 1.4}-\ref{thm 1.5}.

\section{Preliminaries}

In this section,  we review three notions of  upper metric mean dimension with  a  continuous potential on subsets, and establish some  elementary properties of these notions.

Let $C(X, \mathbb{R})$ denote the space of  real-valued continuous functions on $X$. It is a Banach space equipped with the supremum norm $||f||:=\sup_{x\in X}{|f(x)|}$. For  every  $n\in \mathbb{N}$ and  $x\in X$,   we define  the $n$-th Birkhoff  sum of $f$  along the orbit  segment $\{x,Tx,...,T^{n-1}x\}$ of $x$ as $S_nf(x)=\sum_{j=0}^{n-1}f(T^jx)$.  For $x,y\in X$, we denote by $d_n$ the Bowen metric on $X$:
$$d_n(x,y):=\max_{0\leq j\leq n-1} d(T^{j}x,T^jy).$$  Then the  {open   and closed balls} centered at $x$  with radius $\epsilon>0$ in the metric $d_n$  are   respectively given by 
$$B_n(x,\epsilon)=\{y\in X: d_n(x,y)<\epsilon\},$$
$$ \overline B_n(x,\epsilon)=\{y\in X:d_n(x,y)\leq\epsilon\}.$$

For  a non-empty subset $Z\subset X$,  a set $F\subset Z$ is  an \emph{ $(n,\epsilon)$-separated set of $Z$} if $d_n(x,y)\geq\epsilon$ for any  $x,y \in F$ with $x\not= y$. Denote by $s_n(T, Z, d,\epsilon)$ the  largest cardinality  of $(n,\epsilon)$-separated sets of $Z$.  A  set $E\subset X$ is  an  \emph{$(n,\epsilon)$-spanning set of $Z$} if for any $x\in Z$, there exists $y \in E$ such that $d_n(x,y)<\epsilon$.  Denote by $r_n(T,Z, d,\epsilon)$ the   smallest cardinality  of $(n,\epsilon)$-spanning sets of $Z$. 
We put
$$ P_n(T,f, Z,d,\epsilon)=\inf\{\sum_{x\in E}(1/\epsilon)^{S_nf(x)}\},$$
where the infimum  ranges over  all   $(n,\epsilon)$-spanning sets of $Z$. Let
$$P(T,f, Z,d,\epsilon)=\limsup_{n\to \infty} \frac{1}{n} \log P_n(T,f, Z,d,\epsilon),$$
and
$$Q(T,f, Z,d,\epsilon)=\limsup_{n\to \infty} \frac{1}{n} \log \sup\{\sum_{x\in F}(1/\epsilon)^{S_nf(x)}\},$$
where the supremum ranges over   all   $(n,\epsilon)$-separated sets of $Z$.

\begin{prop} \label{prop 2.1}
Let $Z$ be a non-empty subset of $X$ and $f\in C(X, \mathbb{R})$.  
Then for every $0<\epsilon <1$, we have
$$P(T,f, Z,d,\epsilon)\leq Q(T,f, Z,d,\epsilon)$$
and 
$$P(T,f, Z,d,\frac{\epsilon}{2})\geq Q(T,f, Z,d,\epsilon)-||f||\log 2-\gamma(\frac{\epsilon}{2})\logf,$$
where $\gamma(\epsilon)=\sup\{|f(x)-f(y)|:x,y\in X, d(x,y)<\epsilon\}.$
\end{prop}

\begin{proof}
Fix $0<\epsilon <1$. Notice that  an $(n,\epsilon)$-separated set  of $Z$ of  maximum cardinality is also an $(n,\epsilon)$-spanning set of $Z$. Therefore, by the  definitions, one has 
$P(T,f, Z,d,\epsilon)\leq Q(T,f, Z,d,\epsilon)$.

Let $F$ be an $(n,\epsilon)$-separated set of $Z$, and let $E$ be an $(n,\frac{\epsilon}{2})$-spanning set  for $Z$. Choose a map $\Phi: F \rightarrow E$  such that $d_n(x,\Phi(x))<\frac{\epsilon}{2}$ for  every $x\in F$. Then $\Phi$ is  injective.
Hence one has
\begin{align*}
\sum_{x\in E}(2/\epsilon)^{S_nf(x)}&\geq  2^{-n||f||}\sum_{x\in F}(1/\epsilon)^{S_nf(\Phi(x))}\\
&\geq  2^{-n||f||}(1/\epsilon)^{-\gamma(\frac{\epsilon}{2})n}\sum_{x\in F}(1/\epsilon)^{S_nf(x)}.
\end{align*}
It follows that $P(T,f, Z,d,\frac{\epsilon}{2})\geq Q(T,f, Z,d,\epsilon)-||f||\cdot \log 2-\gamma(\frac{\epsilon}{2})\logf$.
\end{proof}

  Motivated by the definition of box dimension in fractal geometry and in view of Proposition \ref{prop 2.1}, we  may equivalently  define the upper  metric mean dimension with a  continuous potential using separated and spanning sets.

\begin{df}\label{df 2.2}
Let $Z\subset X$ be a non-empty subset  and  $f\in C(X,\mathbb{R})$. The upper capacity metric mean dimension  with potential  $f$ on  $Z$ is defined by
\begin{align*}
\overline{\rm mdim}_{M}^{UC}(T,f,Z,d)&= \limsup_{\epsilon\to 0}\frac{P(T,f, Z,d,\epsilon)}{\log \frac{1}{\epsilon}}\\
&= \limsup_{\epsilon\to 0}\frac{Q(T,f, Z,d,\epsilon)}{\log \frac{1}{\epsilon}}.
\end{align*}
\end{df}

\begin{rem}
$(1)$ When $f$ is the zero potential, i.e., $f=0$,   we let $h_{top}^{UC}(T, Z,d,\epsilon):=P(T,0, Z,d,\epsilon)$ denote the upper capacity  $\epsilon$-topological entropy, and  put
$$\overline{\rm mdim}_{M}^{UC}(T,Z,d)= \limsup_{\epsilon\to 0}\frac{h_{top}^{UC}(T, Z,d,\epsilon)}{\log \frac{1}{\epsilon}}.$$ 
If $Z=X$, this notion  exactly  reduces to the  \emph{upper metric mean dimension  of $X$} \cite{lw00}.  We let  $\overline{\rm mdim}_{M}(T,f,X,d):=\overline{\rm mdim}_{M}^{UC}(T,f,X,d)$ and $\overline{\rm mdim}_{M}(T,X,d):=\overline{\rm mdim}_{M}^{UC}(T,X,d)$ in this case.

$(2)$ When $f$ is non-zero and  $Z=X$,  the upper capacity metric mean dimension with potential was  introduced by Tsukamoto \cite{t20} and  was further studied in  \cite{cls21,ycz22b}. Here,  Definition \ref{df 2.2} extends  the concept of metric mean dimension with potential to   subsets of the phase space.

\end{rem}


It is easy to check that the upper capacity metric mean dimension  with potential possesses the following properties: 
\begin{prop}\label{prop 2.4}
Let $(X,d,T)$ be a TDS and $f\in C(X,\mathbb{R})$.

$(1)$ If $Z_1\subset Z_2$, then $$\overline{\rm mdim}_{M}^{UC}(T,f,Z_1,d)\leq \overline{\rm mdim}_{M}^{UC}(T,f,Z_2,d) .$$

$(2)$ If $Z =\bigcup_{1\leq j \leq N}Z_j$ is a finite union of non-empty subsets $Z_1,...,Z_N$, then
\begin{align*}
\overline{\rm mdim}_{M}^{UC}(T,f,Z,d)&= \max_{1\leq j\leq N} \overline{\rm mdim}_{M}^{UC}(T,f,Z_j,d).
\end{align*}

$(3)$ If $Z$ is a non-empty subset of $X$, then  $\overline{\rm mdim}_{M}^{UC}(T,f,Z,d)=\overline{\rm mdim}_{M}^{UC}(T,f,\overline{Z},d)$.

$(4)$ If $T$ is a homeomorphism   and  $Z$ is a non-empty $T$-invariant subset of $X$, then  $$\overline{\rm mdim}_{M}^{UC}(T,f,Z,d)=\overline{\rm mdim}_{M}^{UC}(T^{-1},f,{Z},d).$$
\end{prop}

 Moreover, motivated by Hausdorff and packing dimensions in fractal geometry and by the dimension-type Bowen and packing topological entropies \cite{b73,fh12}, the authors  \cite{cls21,ycz22} gave dimension-theoretic characterizations of metric mean dimensions with potential on subsets using  Carath\'eodory-Pesin structures.  Such formulations  allow us to  invoke  the relevant tools from  dimension theory to characterize the complexity of infinite-entropy systems.

\begin{df}
Let  $Z\subset X$, $f\in C(X, \mathbb{R})$,   $\epsilon>0$, $N\in \mathbb{N}$ and $s \in \mathbb{R}$.  
Put
\begin{align}\label{inequ 2.1}
M_{N,\epsilon}^s(T,f,Z,d)=\inf\{\sum_{i\in I}\limits  e^{-n_i s+\logf \cdot S_{n_i}f(x_i)}\},
\end{align}
where the infimum  is taken over all  finite or countable covers $\{B_{n_i}(x_i,\epsilon)\}_{i\in I}$ of $Z$   with $n_i \geq N$ and $x_i \in X$ for all $i\in I$.

Let $M_{\epsilon}^s(T,f,Z,d)=\lim\limits_{N\to \infty}M_{N,\epsilon}^s(T,f,Z,d)$. There exists  a critical value,  denoted by $M_{\epsilon}(T,f,Z,d)$, at which     $M_{\epsilon}^s(T,f,Z,d)$  jumps from $\infty$ to $0$:
\begin{align*}
M_{\epsilon}(T,f,Z,d):=&\inf\{s:M_{\epsilon}^s(T,f,Z,d)=0\}\\
=&\sup\{s:M_{\epsilon}^s(T,f,Z,d)=\infty\}.
\end{align*}

We define  the Bowen upper metric mean dimension with potential  $f$ on  $Z$ by
\begin{align*}
\overline{\rm {mdim}}_M^B(T,f,Z,d)&=\limsup_{\epsilon \to 0}\frac{M_{\epsilon}(T,f,Z,d)}{\log \frac{1}{\epsilon}}.
\end{align*}
\end{df}

\begin{df}
Let  $Z\subset X$, $f\in C(X, \mathbb{R})$,   $\epsilon>0$, $N\in \mathbb{N}$ and $s \in \mathbb{R}$.  
Put
\begin{align}\label{ienqu 2.2}
P_{N,\epsilon}^s(T,f,Z,d)=\sup\{\sum_{i\in I}\limits  e^{-n_i s+\logf \cdot S_{n_i}f(x_i)}\},
\end{align}
where the supremum is taken over all  finite or countable  pairwise disjoint families of  closed  balls $\{\overline B_{n_i}(x_i,\epsilon)\}_{i\in I}$, with $n_i \geq N $ and $ x_i\in Z$ for all $i\in I$.

Let 
$P_{\epsilon}^s(T,f,Z,d)=\lim_{N \to \infty }\limits P_{N,\epsilon}^s(T,f,Z,d)$. We define
$$\mathcal P_{\epsilon}^s(T,f,Z,d)
=\inf\left\{\sum_{i=1}^{\infty}P_{\epsilon}^s(T,f,Z_i,d): \cup_{i\geq 1}Z_i \supseteq Z \right\}.$$
There is  a critical value, denoted by   $\mathcal P_{\epsilon}(T,f,Z,d)$, at which  $\mathcal P_{\epsilon}^s(T,f,Z,d)$  jumps  from  $\infty$ to $0$:
\begin{align*}
\mathcal P_{\epsilon}(T,f,Z,d):
=&\inf\{s: \mathcal P_{\epsilon}^s(T,f,Z,d)=0\}\\
=&\sup\{s: \mathcal P_{\epsilon}^s(T,f,Z,d)=\infty\}.
\end{align*}

We define the packing upper metric mean dimension  with potential $f$ on  $Z$ by
\begin{align*}
\overline{\rm {mdim}}_M^P(T,f,Z,d)
=\limsup_{\epsilon \to 0}\limits\frac{ \mathcal{P}_{\epsilon}(T,f,Z,d)}{\log \frac{1}{\epsilon}}.
\end{align*}
\end{df}

\begin{rem}
$(1)$ For the zero potential  $f=0$, we set   $h_{top}^B(T,Z,d,\epsilon):=M_{\epsilon}(T,0,Z,d)$ and $h_{top}^P(T,Z,d,\epsilon):=\mathcal P_{\epsilon}(T,0,Z,d)$, and  define  the  \emph{Bowen and packing upper metric mean dimensions on  $Z$} as
\begin{align*}
	\overline{\rm {mdim}}_M^B(T,Z,d)&=\limsup_{\epsilon \to 0}\frac{h_{top}^B(T,Z,d,\epsilon)}{\log \frac{1}{\epsilon}},\\
	\overline{\rm {mdim}}_M^P(T,Z,d)&=\limsup_{\epsilon \to 0}\frac{h_{top}^P(T,Z,d,\epsilon)}{\log \frac{1}{\epsilon}},
\end{align*}
respectively.

$(2)$ To preserve monotonicity, if $Z=\varnothing$, we define each of the above three upper metric mean dimensions to be $0$, and each corresponding upper metric mean dimension with potential to be $-\infty$.

$(3)$ In fact, the quantities in (\ref{inequ 2.1}) and (\ref{ienqu 2.2}) were alternatively defined in \cite{cls21,ycz22} as follows: 
 \begin{align*}
	\hat{M}_{N,\epsilon}^s(T,f,Z,d)&=\inf\{\sum_{i\in I}\limits  e^{-n_i s+\logf \cdot \sup_{y \in B_{n_i}(x_i,\epsilon) } S_{n_i}f(y)}\},\\
	\hat{P}_{N,\epsilon}^s(T,f,Z,d)&=\sup\{\sum_{i\in I}\limits  e^{-n_i s+\logf \cdot  \sup_{y \in \overline{B}_{n_i}(x_i,\epsilon) } S_{n_i}f(y)}\}.
\end{align*}  
Using  the following two inequalities   
\begin{align*}
|\sup_{y \in B_{n}(x,\epsilon) } S_{n}f(y)-S_{n}f(x)|&\leq n\gamma(\epsilon),\\
|\sup_{y \in \overline{B}_{n}(x,\epsilon) } S_{n}f(y)-S_{n}f(x)|&\leq n\gamma(2\epsilon),
\end{align*}
for  every $n$ and $x\in X$,
and the fact that $\gamma(\epsilon) \to 0$ as $\epsilon \to 0$,
one can show that the present formulations are equivalent to those in  \cite{cls21,ycz22}, and hence yield the same values of $\overline{\rm {mdim}}_M^B(T,f,Z,d)$ and $\overline{\rm {mdim}}_M^P(T,f,Z,d)$.  We use the present definitions because they simplify calculations for certain fractal-like sets.

$(4)$ Unlike topological pressure, these upper metric mean dimensions with potential depend on the chosen metric on $X$, but they are unchanged when the metric is replaced by a bi-Lipschitz equivalent one.
\end{rem}

Next, we  clarify the precise  relations  among  the three  different types of  metric mean dimension with potential.

\begin{prop} \label{prop 2.6}
Let $(X,d, T)$ be a  TDS and  $f\in C(X, \mathbb{R})$.   Then
\begin{enumerate}
\item If $Z_1\subset Z_2 \subset X$, then $\overline{\rm {mdim}}_M^B(T,f,Z_1,d)\leq  \overline{\rm {mdim}}_M^B(T,f,Z_2,d)$,\\ and  $\overline{\rm {mdim}}_M^P(T,f,Z_1,d)\leq  \overline{\rm {mdim}}_M^P(T,f,Z_2,d)$.
\item If  $Z=\cup_{1\leq j \leq N} Z_j$,  then $\overline{\rm {mdim}}_M^B(T,f,Z,d)=\max_{1 \le j \le N} \limits \overline{\rm {mdim}}_M^B(T,f,Z_j,d),$\\ and  $\overline{\rm {mdim}}_M^P(T,f,Z,d)=\max_{1 \le j \le N} \limits \overline{\rm {mdim}}_M^P(T,f,Z_j,d)$.
\item If  $Z$  is  a non-empty subset of $X$, then for every $0<\epsilon<1$, 
\begin{align*}
M_{3\epsilon}(T,f,Z,d)&\leq \mathcal{P}_{\epsilon}(T,f,Z,d)+\log 3\cdot ||f||,\\
\mathcal{P}_{\epsilon}(T,f,Z,d)&\leq Q(T,f, Z,d,\epsilon),
\end{align*}
and hence
$\overline{\rm {mdim}}_M^B(T,f,Z,d)\leq \overline{\rm {mdim}}_M^P(T,f,Z,d)\leq \overline{\rm mdim}_{M}^{UC}(T,f,Z,d).$
\item  If $Z$ is a  $T$-invariant compact subset of $X$,  then  for every $0<\epsilon<1$, 
$$P(T,f,Z,d,\epsilon)\leq   ||f||\log 2+  \gamma(\frac{\epsilon}{2})\log\frac{2}{\epsilon}+M_{\frac{\epsilon}{2}}(T,f,Z,d),$$
and hence
$\overline{\rm {mdim}}_M^B(T,f,Z,d)= \overline{\rm {mdim}}_M^P(T,f,Z,d)= \overline{\rm mdim}_{M}^{UC}(T,f,Z,d).$
\end{enumerate}
\end{prop}
\begin{proof}
 $(1)$ and $(2)$  follow directly from  the definitions.

$(3)$
Fix $0<\epsilon<1$. Let $N\in \mathbb{N}$, $s\in \mathbb{R}$, and let $A$ be a non-empty subset of $X$. By compactness of $X$, choose  a  disjoint family  $\{\overline{B}_N(x_i,\epsilon): x_i\in A,1\leq i\leq R\}$ of maximum  cardinality. Then $\bigcup_{1\leq i\leq R} {B}_N(x_i,3\epsilon)\supset A$. Hence, one has
\begin{align*}
M_{N,3\epsilon}^s(T,f,A,d) &\leq \sum_{i=1}^{R}\limits  e^{-N s+\log\frac{1}{3\epsilon} \cdot S_{N}f(x_i)} \\
&\leq \sum_{i=1}^{R}\limits  e^{-N (s-\log 3\cdot ||f||)+\log\frac{1}{\epsilon} \cdot S_{N}f(x_i)} \\
&\leq P_{N,\epsilon}^{s-\log 3\cdot ||f||}(T,f,A,d).
\end{align*}
Thus, whenever $Z\subset \cup_{i\geq 1}Z_i$, we have $$M_{3\epsilon}^s(T,f,Z,d)\leq \sum_{i\geq 1}M_{3\epsilon}^s(T,f,Z_i,d)\leq  \sum_{i\geq 1} P_{\epsilon}^{s-\log 3\cdot ||f||}(T,f,Z_i,d),$$ and hence $M_{3\epsilon}^s(T,f,Z,d)\leq  \mathcal{P}_{\epsilon}^{s-\log 3\cdot ||f||}(T,f,Z,d).$
This implies that $$M_{3\epsilon}(T,f,Z,d)\leq \mathcal{P}_{\epsilon}(T,f,Z,d)+\log 3\cdot ||f||.$$

Let $-\infty <t<s< \mathcal P_{\epsilon}(T,f,Z,d)$.  Then  there exists $N_0$ such that for every $N\geq N_0$,  we can choose a  finite or countable  pairwise disjoint   family  of closed  Bowen balls $\{\overline B_{n_i}(x_i,\epsilon)\}_{i\in I}$   with $n_i \geq N $ and $ x_i\in Z $,  such that 
$$\sum_{i \in I} e^{-n_i s+\logf \cdot S_{n_i}f(x_i)}= \sum_{n\geq N} \sum_{i\in I_n}e^{-n s+\logf \cdot S_{n}f(x_i)}>1,$$
where $I_n=\{i\in I: n_i=n\}$. Therefore, there exists $n\geq N$ (depending on $N$) such that
$$\sum_{i\in I_n}(1/\epsilon)^{ S_{n}f(x_i)}>(1-e^{t-s})e^{nt}.$$
Notice that $\{x_i:i\in I_n\}$ is an  $(n,\epsilon)$-separated set of $Z$.  Taking logarithms, dividing by $n$, and then passing to the upper limit yields $Q(T,f,Z,d,\epsilon)\ge t$. Letting $t\to \mathcal P_{\epsilon}(T,f,Z,d) $, we get the desired inequality.

(4)  Fix $0<\epsilon<1$. Let $\alpha>M_{\frac{\epsilon}{2}}(T,f,Z,d)$. Then $M_{\frac{\epsilon}{2}}^\alpha(T,f,Z,d)<1$. 
Since $Z$ is compact,  there exist $N_0\in \mathbb{N}$  and a  finite open cover  $\{B_{n_i}(x_i,\frac{\epsilon}{2})\}_{i\in I}$ of $Z$ with $n_i \geq N_0$ and $x_i  \in  X $ for every $i\in I$,  such that
\begin{align}\label{ineq 2.3}
\sum_{i\in I}\limits  e^{-n_i \alpha+\log\frac{2}{\epsilon} \cdot S_{n_i}f(x_i)}<1.
\end{align}
It follows from (\ref{ineq 2.3}) that
\begin{align}\label{inequ 2.4}
K:=\sum_{k\geq 1}\sum_{i_1,...,i_k\in I}\limits  e^{-(n_{i_1}+\cdots n_{i_k}) \alpha+\log\frac{2}{\epsilon}\cdot \sum_{j=1}^kS_{n_{i_j}}f(x_{i_j})}<\infty.
\end{align}
For every $N\geq \max\{N_0,2M\}$, where  $M=\max_{i\in I}n_i$,  we define the family 
$$\mathcal{F}_N:=\left\{ B_{n_{i_1}}(x_{i_1},\frac{\epsilon}{2}) \bigcap \left(\bigcap _{k=2}^t    T^{-(n_{i_1}+\cdots n_{i_{k-1}})}B_{n_{i_k}}(x_{i_k}, \frac{\epsilon}{2})\right): t\geq 2, N\leq \sum_{k=1}^tn_{i_k}<N+M, i_k\in I, 1\leq k\leq t\right\}.$$
We  may  assume that  each element of  $\mathcal{F}_N$ is non-empty. Choose $x_A\in A$ for each $A \in  \mathcal{F}_N$. By the $T$-invariance of $Z$,  one has $Z\subset \cup_{A\in \mathcal{F}_N} B_N(x_A,\epsilon)$.   Then the set $E_N:=\{x_A: A \in  \mathcal{F}_N\}$  is an $(N,\epsilon)$-spanning set for $Z$. Using the construction of $\mathcal{F}_N$, we estimate the following term:
\begin{align}\label{inequ 2.3}
&e^{-(\alpha +\gamma(\frac{\epsilon}{2})\log\frac{2}{\epsilon})N}\sum_{A\in \mathcal{F}_N}(\frac{1}{\epsilon})^{S_Nf(x_A)} \nonumber \\
\leq &2^{N||f||} e^{-(\alpha+ \gamma(\frac{\epsilon}{2})\log\frac{2}{\epsilon})N}\sum_{A\in \mathcal{F}_N}e^{\log\frac{2}{\epsilon}\cdot S_Nf(x_A)} \nonumber \\
\leq& \max\{1, e^{(\alpha+ \gamma(\frac{\epsilon}{2})\log\frac{2}{\epsilon})M}\} 2^{N||f||}  \cdot 
\sum_{A\in \mathcal{F}_N}e^{-(\alpha+ \gamma(\frac{\epsilon}{2})\log\frac{2}{\epsilon})\sum_{k=1}^{t_A}n_{i_k}+\log\frac{2}{\epsilon}\cdot S_Nf(x_A)} \nonumber \\
\leq & C_1  \cdot 2^{N||f||} 
\sum_{A\in \mathcal{F}_N} e^{\sum_{k=1}^{t_A}-(\alpha+ \gamma(\frac{\epsilon}{2})\log\frac{2}{\epsilon})\cdot n_{i_k}+ \log\frac{2}{\epsilon}\cdot S_{n_{i_1}+\cdots+n_{i_{t_A}}} f(x_A)}, 
\end{align}
where $C_1:= \max\{1, e^{(\alpha+ \gamma(\frac{\epsilon}{2})\log\frac{2}{\epsilon})M}\}(\frac{2}{\epsilon})^{M||f||}.$ By the choice of $x_A$,  we have
\begin{align}\label{ineq 2.6}
&e^{\sum_{k=1}^{t_A}\{-(\alpha+ \gamma(\frac{\epsilon}{2})\log\frac{2}{\epsilon})\cdot n_{i_k}\}+ \log\frac{2}{\epsilon}\cdot S_{n_{i_1}+\cdots+n_{i_{t_A}}} f(x_A)} \nonumber \\
=& e^{\sum_{k=1}^{t_A}\{-(\alpha+ \gamma(\frac{\epsilon}{2})\log\frac{2}{\epsilon})\cdot n_{i_k}+ \log\frac{2}{\epsilon}\cdot S_{n_{i_k}} f(T^{n_{i_1}+\cdots+n_{i_{k-1}}}x_A)\}} \nonumber \\
\leq &e^{\sum_{k=1}^{t_A}\{-\alpha n_{i_k}+ \log\frac{2}{\epsilon}\cdot S_{n_{i_k}} f(x_{i_k})\}} \nonumber \\
=&e^{-(n_{i_1}+\cdots n_{i_{t_A}}) \alpha+\log\frac{2}{\epsilon}\cdot \sum_{j=1}^{t_A}S_{n_{i_j}}f(x_{i_j})}. 
\end{align}
By  (\ref{inequ 2.4}), (\ref{inequ 2.3}), and (\ref{ineq 2.6}),  we get 
\begin{align}
e^{-(\alpha +\gamma(\frac{\epsilon}{2})\log\frac{2}{\epsilon})N}\sum_{A\in \mathcal{F}_N}(\frac{1}{\epsilon})^{S_Nf(x_A)}  
\leq  C_1  K \cdot 2^{N||f||}. 
\end{align}
This   implies that  $P_N(T,f, Z,d,\epsilon)\leq  C_1  K \cdot 2^{N||f||} e^{(\alpha +\gamma(\frac{\epsilon}{2})\log\frac{2}{\epsilon})N},$ where $C_1$ and $K$ are   constants   independent of  $N$. Then 
$$P(T,f, Z,d,\epsilon)\leq ||f||\log 2+ \alpha +\gamma(\frac{\epsilon}{2})\log\frac{2}{\epsilon}.$$
Letting $\alpha \to M_{\frac{\epsilon}{2}}(T,f,Z,d)$ and  then applying  Proposition \ref{prop 2.1} yields the desired  result.
\end{proof}
\begin{rem}\label{rem 2.9}
$(1)$  We  remark that  $\overline{\rm {mdim}}_M^B(T,f,Z,d)= \overline{\rm mdim}_{M}^{UC}(T,f,Z,d)$ was first proved in  \cite[Proposition 2.2 (2)]{cls21} using   the language of open covers.  Our proof is direct   and uses  Bowen balls.

$(2)$ The dimensional conclusions in Proposition  \ref{prop 2.6} (3) and (4) appear in  \cite[Proposition 3.4]{ycz22}, but that result does not include the finer estimates for these pressure-like quantities at scale $\epsilon$.

$(3)$  Bowen's approach  \cite{b73} to Bowen topological entropy can also  be used to define  the  Bowen metric mean dimension  on   subsets  (cf. Proposition \ref{prop 7.1}).
\end{rem}

\section{Metric mean dimensions of $\epsilon$-stable sets and  their variants} \label{sec 3}

\subsection{Metric mean dimension of blocks of   $\epsilon$-stable sets}
Using the finite-union  stability of   upper metric mean dimension  with potential on subsets, we  prove Theorem \ref{thm 1.1}.

\begin{proof}[Proof of Theorem \ref{thm 1.1}]
Fix a finite subset $\mathcal{S}^{*} \subset \mathbb{N}$ and $\epsilon >0$.  The continuity of $T$ implies  that the set   $W_{\epsilon,  \mathcal{S}^{*}}^s(x,T)$  is a  closed  neighborhood of $x$. By  the compactness of $X$, we can   choose  a  finite family $\{ W_{\epsilon,   \mathcal{S}^{*}}^s(x_j,T)\}_{j=1}^m$ to cover $X$.  Applying Propositions \ref{prop 2.4} and \ref{prop 2.6},  we obtain
\begin{align}\label{inequ 3.3.1}
\begin{split}
\overline{\rm mdim}_M^B(T,f,X,d)&=\max_{x\in X}\overline{\rm mdim}_M^B(T,f,W_{\epsilon,   \mathcal{S}^{*}}^s(x,T),d)\\
&\leq \max_{x\in X}\overline{\rm mdim}_M^P(T,f,W_{\epsilon,   \mathcal{S}^{*}}^s(x,T),d)\\
&\leq \max_{x\in X}\overline{\rm mdim}_M^{UC}(T,f,W_{\epsilon,   \mathcal{S}^{*}}^s(x,T),d)=\overline{\rm mdim}_M(T,f,X,d).
\end{split}
\end{align}
\end{proof}

\begin{rem}
The $T$-invariance assumption in Proposition \ref {prop 2.6} (4) is sufficient but not necessary: the proof of Theorem \ref{thm 1.1} exhibits closed sets, not necessarily $T$-invariant, for which the three types of upper metric mean dimension with potential coincide.
\end{rem}

\begin{cor}
Let $(X,d,T)$ be a TDS,  and let $\mathcal{S}^{*}$  be a finite subset of $\mathbb{N}$.  Then  for every $\epsilon >0$ and $Q\in \{UC, B,P\}$,
\begin{align*}
\overline{\rm mdim}_M(T,X,d)&=\max_{x\in X}\overline{\rm mdim}_M^Q(T,W^s_{\epsilon,   \mathcal{S}^{*}}(x,T),d).
\end{align*} 
\end{cor}

\subsection{Metric mean dimension of $\epsilon$-stable sets and their preimages}

We next  prove Theorems \ref{thm 1.2} and \ref{thm 1.3}. 

The   proof below draws on ideas from   \cite{bow72,b73,o911,wa22}. The central idea is to construct suitable families of sets via a pigeonhole argument and then obtain an $(n,\delta)$-spanning set for the Bowen ball $B_n(x,\epsilon)$.

\begin{lem}\label{lem 3.5}
Let $(X,d,T)$ be a TDS  and $f\in C(X, \mathbb{R})$. Fix $\epsilon >0$. For  every $0<\delta <1$ and  $\lambda> \sup_{x\in X} M_{\frac{\delta}{2}}(T,f,W_\epsilon^s(x,T),d)$, there exist $N \in \mathbb{N}$  and  a constant $M$ such that for  all $n> N$ and $x\in X$,
$$ P_n(T,f, \overline{B}_n(x,\epsilon),d,\delta)\leq M\cdot 2^{n||f||} (2/\delta)^{n\gamma(\delta)} \cdot e^{\lambda n}.$$

\end{lem}

\begin{proof}
Fix $0<\delta <1$. By the definitions and Proposition \ref{prop 2.6} (3), we have
$$\sup_{x\in X} M_{\delta/2}\big(T,f,W_\varepsilon^s(x,T),d\big) \leq M_{\delta/2}\big(T,f,X,d\big) \leq Q(T,f, X,d,  \frac{\delta}{6})+\log 3\cdot ||f||.$$
Together with the inequality  $$Q(T,f, X,d,\frac{\delta}{6})\leq h_{top}^{UC}(T, X,d,\frac{\delta}{6})+||f|| \log \frac{6}{\delta},$$
we know that $\sup_{x\in X} M_{\delta/2}\big(T,f,W_\varepsilon^s(x,T),d\big)$
is  always finite for every $\delta >0$. Now we let $\lambda> \sup_{x\in X} M_{\frac{\delta}{2}}(T,f,W_\epsilon^s(x,T),d)$. Then for all $y\in X$,  one has $$M_{\frac{\delta}{2}}^\lambda(T,f,W_\epsilon^s(y,T),d)<1.$$  By the compactness of  $W_\epsilon^s(y,T)$,  there  exists
a finite open cover $\{B_{n_i(y)}(x_i(y),\frac{\delta}{2})\}_{i=1}^{k(y)}$ of $ W_\epsilon^s(y,T)$  with $n_i(y)\geq 1 $ and $ x_i(y)\in X$  for all $1\leq i\leq k(y)$, such  that 
\begin{align}\label{inequ 3.2}
M(y):=\sum_{1\leq i \leq k(y)}\limits  e^{-n_i(y) \lambda+\log\frac{2}{\delta} \cdot S_{n_i(y)}f(x_i (y))}<1.
\end{align}
Notice that $W_\epsilon^s(y,T)=\cap_{n\geq 1}\overline{B}_n(y,\epsilon)$. One can choose $m(y)\in \mathbb{N}$ satisfying  $\overline{B}_{m(y)}(y,\epsilon)\subset \bigcup_{1\leq i\leq k(y)} B_{n_i(y)}(x_i(y),\frac{\delta}{2})$. Since $\overline{B}_{m(y)}(y,\epsilon)=\cap_{\gamma >\epsilon} \overline{B}_{m(y)}(y,\gamma)$, there exists  $\gamma >\epsilon$ such that $$\overline{B}_{m(y)}(y,\gamma)\subset \bigcup_{1\leq i\leq k(y)} B_{n_i(y)}(x_i(y),\frac{\delta}{2}).$$
Let $U_y:=\{z\in X: d_{m(y)}(z,y)<\gamma -\epsilon\}$ denote the open neighborhood of $y$. Then  for every $y\in X$, one has
\begin{align}\label{eq 3.2}
\bigcup_{z \in U_y} \overline{B}_{m(y)}(z,\epsilon)\subset \bigcup_{1\leq i\leq k(y)} B_{n_i(y)}(x_i(y),\frac{\delta}{2}).
\end{align}
By the compactness of $X$,  choose $y_1,...,y_p$ such that $\{U_{y_1},...,U_{y_p}\}$ is   a finite open cover of $X$.   Let 
$$N=\max\{m(y_j),n_i(y_j):1\leq j \leq p, 1\leq i\leq k(y_j)\},~K:=\max_{1\leq j \leq p}k(y_j),$$
and  let $E$ be an  $(N,\frac{\delta}{2})$-spanning set of $X$ with minimum  cardinality.   For every $x\in X$, we fix $\phi(x)\in \{1,...,p\}$ such that  $x\in U_{y_{\phi(x)}}$. 

For each  $t\in \mathbb{N}$,  let  $(j_1,...,j_t)$ be  a   tuple  of positive integers with $1\leq j_l \leq K$ for every $1\leq l\leq t$. Set $a(1)=0$ and define $a(l)$ recursively by
$$ a(l)=a(l-1)+n_{j_{l-1}}(y_{\phi (T^{a(l-1)}x)})$$
 for $2\le l\le t+1$. We say that $(j_1,\dots,j_t)$ is admissible with respect to $x\in X$ if $1\le j_l\le k(y_{{\phi(T^{a(l)}x)}})$ for every $1\le l\le t$. Define $A_t(x)$  to be the set of  tuples of length $t$  admissible with $x$, and for every $n\in\mathbb{N}$,  put $$\mathcal{F}_n(x):=\{(j_1,...,j_t):~(j_1,...,j_t)\in A_t(x), ~ t\geq 1, ~a(t)+N<n\leq a(t+1)+N\}.$$  For every $(j_1,...,j_t) \in  \mathcal{F}_n(x)$, we define the set 
$$E(x,j_1,...,j_t):=\mathop\bigcap_{l=1}^t T^{-a(l)}B_{n_{j_l} (y_{\phi (T^{a(l)}x)}) }(x_{j_l}(y_{\phi (T^{a(l)}x)}),\frac{\delta }{2}).$$
For  every $n> N$, if  $E(x,j_1,...,j_t) \cap T^{-a(t+1)} B_N(z,\frac{\delta}{2})\not=\varnothing
$ for some $z\in E$, then we may choose $ y_{(j_1,...,j_t,z)}$ such that 
$$E(x,j_1,...,j_t)\cap T^{-a(t+1)} B_N(z,\frac{\delta}{2}) \subset B_{n}(y_{(j_1,...,j_t,z)},\delta),$$
since  $n\leq a(t+1)+N$.

We claim that, for every $n  >N$ and $x\in X$, 
\begin{align*}
	\overline{B}_n(x,\epsilon) &\subset \bigcup_{(j_1,...,j_t)\in  \mathcal{F}_n(x)} \bigcup_{z\in E} (E(x,j_1,...,j_t)\cap  T^{-a(t+1)} B_N(z,\frac{\delta}{2}))\\
	&\subset \bigcup_{(j_1,...,j_t)\in  \mathcal{F}_n(x)}\bigcup_{z\in E} B_{n}(y_{(j_1,...,j_t,z)},\delta).
\end{align*}
Indeed,  if $y \in  \overline{B}_n(x,\epsilon)$, then, by (\ref{eq 3.2}), there exists $1\leq j_1 \leq k(y_{\phi(x)})$ such that $$y\in  B_{n_{j_1} (y_{\phi (x)}) }(x_{j_1}(y_{\phi (x)}),\frac{\delta }{2}).$$ If $n-n_{j_1} (y_{\phi (x)})>N$, then $T^{a(2)}y\in \overline{B}_{n-a(2)}(T^{a(2)}x,\epsilon)$. Applying (\ref{eq 3.2}) again, we find  $1\leq j_2 \leq k(y_{\phi(T^{a(2)}x)})$ such that $$T^{a(2)}y\in  {B}_{n_{j_2} (y_{\phi (T^{a(2)}x)}) }(x_{j_2}(y_{\phi (T^{a(2)}x)}),\frac{\delta }{2});$$ otherwise,  choose $z\in E$  such that $T^{a(2)}y\in  B_N(z,\frac{\delta}{2})$. Continue this process as long as $n-a(t)>N$, and stop once $n-a(t+1)\le N$. Suppose  that, for $y$,  the process    terminates after $t+1$ steps. Then   for some $z\in E$, 
$$y\in E(x,j_1,...,j_t)\cap T^{-a(t+1)} B_N(z,\frac{\delta}{2}),$$
which implies the desired claim.

Consequently, the set  $\{y_{(j_1,...,j_t,z)}:(j_1,...,j_t)\in\mathcal{F}_n(x), z\in E\}$ is an $(n,\delta)$-spanning set of $\overline{B}_n(x,\epsilon)$. To simplify the calculation, we assume that  $E(x,j_1,...,j_t) \cap T^{-a(t+1)} B_N(z,\frac{\delta}{2})\not=\varnothing$  for every $(j_1,...,j_t)\in  \mathcal{F}_n(x)$ and $z\in E$; otherwise,  one can discard  the corresponding empty intersections and then obtain a ``smaller" $(n,\delta)$-spanning set  of $\overline{B}_n(x,\epsilon)$. Hence, 
\begin{equation} \label{eq 3.3}
\begin{aligned}
e^{-\lambda \cdot n}\sum_{y_{(j_1,...,j_t,z)}}(1/\delta)^{S_nf(y_{(j_1,...,j_t,z)})}
= &e^{-\lambda \cdot n}\sum_{y_{(j_1,...,j_t,z)}}e^{\log \frac{1}{\delta}\cdot {S_nf(y_{(j_1,...,j_t,z)})}}\\
\leq & 2^{n||f||}\cdot\sum_{y_{(j_1,...,j_t,z)}}e^{{-\lambda  \cdot n}+\log \frac{2}{\delta}\cdot{S_nf(y_{(j_1,...,j_t,z)})}}\\
\leq &\max\{e^{-\lambda N},1\}\cdot  2^{n||f||}\sum_{y_{(j_1,...,j_t,z)}}e^{{-\lambda}a(t+1)+\log \frac{2}{\delta}\cdot{S_nf(y_{(j_1,...,j_t,z)})}}.
\end{aligned}
\end{equation}
Notice that $a(t+1)+N\geq n>a(t+1)$. Then 
\begin{align}\label{eq 3.4}
S_nf(y_{(j_1,...,j_t,z)})&=S_{a(t+1)}f(y_{(j_1,...,j_t,z)})+S_{n-a(t+1)}f(T^{a(t+1)}y_{(j_1,...,j_t,z)}) \nonumber\\
&\leq  S_{a(t+1)}f(y_{(j_1,...,j_t,z)})+N||f||.
\end{align} 
Using  this inequality, we have 
\begin{equation} \label{eq 3.5}
\begin{aligned}
e^{-\lambda \cdot n}\sum_{y_{(j_1,...,j_t,z)}}(1/\delta)^{S_nf(y_{(j_1,...,j_t,z)})}
\leq   C\cdot 2^{n||f||}\sum_{y_{(j_1,...,j_t,z)}}e^{{-\lambda}a(t+1)+\log \frac{2}{\delta}\cdot{S_{a(t+1)}f(y_{(j_1,...,j_t,z)})}},
\end{aligned}
\end{equation}
where $C:=\max\{e^{-\lambda N},1\}(2/\delta)^{N||f||}$ is a constant. It remains  to estimate:
$$e^{{-\lambda}a(t+1)+\log \frac{2}{\delta}\cdot{S_{a(t+1)}f(y_{(j_1,...,j_t,z)})}}.$$ By   the choice  of  $y_{(j_1,...,j_t,z)}$,  for every  $1\leq l\leq t$ we have 
\begin{align*}
|S_{n_{j_l} (y_{\phi (T^{a(l)}x)})}f(x_{j_l}(y_{\phi (T^{a(l)}x)}))- S_{n_{j_l} (y_{\phi (T^{a(l)}x)})}f(T^{a(l)}y_{(j_1,...,j_t,z)})|
\leq {n_{j_l} (y_{\phi (T^{a(l)}x)})}\gamma(\delta).
\end{align*}
This yields  that
\begin{equation}\label{inequ 3.7}
\begin{aligned}
&e^{{-\lambda}a(t+1)+\log \frac{2}{\delta}\cdot{S_{a(t+1)}f(y_{(j_1,...,j_t,z)})}}\\
\leq& (2/\delta)^{a(t+1)\gamma(\delta)}e^{\sum_{l=1}^t\{-\lambda \cdot n_{j_l} (y_{\phi (T^{a(l)}x)})+\log \frac{2}{\delta}\cdot S_{n_{j_l} (y_{\phi (T^{a(l)}x)})}f(x_{j_l}(y_{\phi (T^{a(l)}x)}))\} }\\
\leq& (2/\delta)^{n\gamma(\delta)}e^{\sum_{l=1}^t\{-\lambda \cdot n_{j_l} (y_{\phi (T^{a(l)}x)})+\log \frac{2}{\delta}\cdot S_{n_{j_l} (y_{\phi (T^{a(l)}x)})}f(x_{j_l}(y_{\phi (T^{a(l)}x)}))\} }.
\end{aligned}
\end{equation}

Finally, by  (\ref{eq 3.5}) and (\ref{inequ 3.7}), we obtain
\begin{equation}
\begin{aligned}
&e^{-\lambda  \cdot n}\sum_{y_{(j_1,...,j_t,z)}}(1/\delta)^{S_nf(y_{(j_1,...,j_t,z)})}\\
\leq& 
C\cdot 2^{n||f||} (2/\delta)^{n\gamma(\delta)} \cdot \#E\cdot \sum_{(j_1,...,j_t)\in\mathcal{F}_n(x)} e^{\sum_{l=1}^t\{-\lambda \cdot n_{j_l} (y_{\phi (T^{a(l)}x)})+\log \frac{2}{\delta}\cdot S_{n_{j_l} (y_{\phi (T^{a(l)}x)})}f(x_{j_l}(y_{\phi (T^{a(l)}x)}))\} }\\
\leq & C\cdot 2^{n||f||} (2/\delta)^{n\gamma(\delta)} \cdot \#E\cdot \sum_{t\geq 1}(\max_{1\leq l\leq p}M(y_l))^t~\text{by}~(\ref{inequ 3.2})\\
\leq & M\cdot 2^{n||f||} (2/\delta)^{n\gamma(\delta)},
\end{aligned}
\end{equation}
where $M:=C\cdot \#E\cdot \sum_{t\geq 1}(\max_{1\leq l\leq p}M(y_l))^t<\infty.$ This  completes the proof.
\end{proof}

We now prove Theorem \ref{thm 1.2} using  Lemma \ref{lem 3.5}.

\begin{proof}[Proof of Theorem \ref{thm 1.2}]
Notice that   $W_{\epsilon}^s(x,T) \subset W_{\epsilon,   \mathcal{S}}^s(x,T)$ for every infinite subset $\mathcal{S}$ of $\mathbb{N}$. By monotonicity of upper metric mean dimension with potential, it suffices  to show that Theorem   \ref{thm 1.2}  holds for  $W_{\epsilon}^s(x,T)$. It follows from Propositions  \ref{prop 2.1} and \ref{prop 2.6} that
\begin{equation}
\begin{aligned}
&\limsup_{\delta\to 0}\frac{1}{\log \frac{1}{\delta}}\sup_{x\in X}M_{\delta}(T,f,W_\epsilon^s(x,T),d)\\
\leq&
\limsup_{\delta\to 0}\frac{1}{\log \frac{1}{\delta}}\sup_{x\in X} \mathcal{P}_{\delta}(T,f,W_\epsilon^s(x,T),d)\\
\leq&
\limsup_{\delta\to 0}\frac{1}{\log \frac{1}{\delta}}\sup_{x\in X} P(T,f, W_\epsilon^s(x,T),d,\delta)\\
\leq&  \overline{\rm mdim}_M(T,f,X,d).
\end{aligned}
\end{equation}

It remains to prove the  reverse  inequality:
\begin{align}\label{eq 3.9}
\overline{\rm mdim}_M(T,f,X,d)\leq  \limsup_{\delta\to 0}\frac{1}{\log \frac{1}{\delta}}\sup_{x\in X}M_{\delta}(T,f,W_\epsilon^s(x,T),d). 
\end{align}
Fix   $0<\epsilon <1$, $0<\delta <1$, and  $\lambda> \sup_{x\in X} M_{\frac{\delta}{2}}(T,f,W_\epsilon^s(x,T),d)$ as in Lemma \ref{lem 3.5}. Then   there exist  $N \in \mathbb{N}$  and  a constant $M$ such that 
$$ P_n(T,f, \overline{B}_n(x,\epsilon),d,\delta)\leq M\cdot 2^{n||f||} (2/\delta)^{n\gamma(\delta)}  e^{\lambda n}$$
for  all $n> N$ and $x\in X$.
Let  $n> N$, and let $E$ be an  $(n,\epsilon)$-spanning set of  $X$ of  minimum cardinality $r_n(T,X,d,\epsilon)$. Then 
\begin{align}
P_n(T,f, X,d,\delta)&\leq \sum_{x\in E} P_n(T,f, \overline{B}_n(x,\epsilon),d,\delta)\nonumber\\
& \leq r_n(T,X,d,\epsilon)\cdot  M\cdot 2^{n||f||} (2/\delta)^{n\gamma(\delta)}  e^{\lambda n}.
\end{align}
This  implies that  $P(T,f, X,d,\delta)\leq  h_{top}^{UC}(T, X,d,\epsilon)+||f||\log 2+\gamma(\delta)\log \frac{2}{\delta}+\lambda.$ Letting  $\lambda \to  \sup_{x\in X} M_{\frac{\delta}{2}}(T,f,W_\epsilon^s(x,T),d)$,
dividing  by  $\log \frac{1}{\delta}$, and then taking $\limsup_{\delta \to 0}$,  we obtain  (\ref{eq 3.9}).
\end{proof}

The following  corollary shows how the divergence rates of topological $\epsilon$-entropy-like quantities of $\epsilon$-stable sets are related to the metric mean dimension of $X$.

\begin{cor}\label{coro 1.4}
Let $(X,d,T)$ be a TDS,  and let $\mathcal{S}$  be an infinite subset  of $\mathbb{N}$.  Then for  any $\epsilon >0$ and $Q\in \{UC,B,P\}$,
\begin{align*}
\overline{\rm mdim}_M(T,X,d)=\limsup_{\delta\to 0}\frac{1}{\log \frac{1}{\delta}}\sup_{x\in X}h_{top}^Q(T,W_{\epsilon, \mathcal{S}}^s(x,T),d,\delta).
\end{align*}   
\end{cor}

\begin{rem}\label{rem 3.5}
After completing this paper, we  learned  that  Tsukamoto \cite[Theorem 1.1]{tsu22}    established  that
$$\overline{\rm mdim}_M(T,X,d)=\limsup_{\delta\to 0}\frac{1}{\log \frac{1}{\delta}}\sup_{x\in X}h_{top}^{UC}(T,W_{\epsilon}^{s,*}(x,T),d,\delta)$$
when  $T: X\rightarrow X$  is a homeomorphism   and  $$W_\epsilon^{s,*}(x,T)=\{y\in X:d(T^jx,T^jy)\leq\epsilon~ \forall j\in \mathbb{Z}\}$$ denotes the   (two-sided) $\epsilon$-stable set of $x$. 
\end{rem}

By Theorem~\ref{thm 1.2}, we obtain an analogous equality for the divergence rate of the $\epsilon$-pressure associated with the dispersion of preimages of $\epsilon$-stable sets.

\begin{proof}[Proof of  Theorem \ref{thm 1.3}]
Notice that  $W_\epsilon^s(x,T) \subset T^{-n} W_\epsilon^s(T^nx,T)$ holds for  all $n$ and $x\in X$. Then, by Theorem \ref{thm 1.2}, one has
\begin{align*}
\overline{\rm mdim}_M(T,f,X,d)&\leq\limsup_{\delta\to 0}\frac{1}{\log \frac{1}{\delta}} GR( \{\sup_{x\in X}P_n(T,f, W_\epsilon^s(x,T),d,\delta)\}_{n\geq 1}) \\
&\leq\limsup_{\delta\to 0}\frac{1}{\log \frac{1}{\delta}}GR( \{\sup_{x\in X}P_n(T, f,T^{-n}W_\epsilon^s(T^nx,T),d,\delta)\}_{n\geq 1})\\
&\leq\limsup_{\delta\to 0}\frac{1}{\log \frac{1}{\delta}}GR( \{\sup_{y\in X}P_n(T, f,T^{-n}W_\epsilon^s(y,T),d,\delta)\}_{n\geq 1})\\
& \leq \overline{\rm mdim}_M(T,f,X,d).
\end{align*}
\end{proof}

\begin{cor}
Let $(X,d,T)$ be a TDS.  Then  for every $\epsilon >0 $, 
\begin{align*}
\overline{\rm mdim}_M(T,X,d)=\limsup_{\delta\to 0}\frac{1}{\log \frac{1}{\delta}} \limsup_{n\to \infty} \frac{1}{n}\sup_{x\in X}  \log r_n(T, T^{-n}W_\epsilon^s(x,T),d,\delta).
\end{align*}   
\end{cor}

Finally, we  present an example to clarify the  aforementioned results for $\epsilon$-stable sets.
\begin{ex}\label{ex 3.7}
Let $[0,1]$ be   the unit interval with the standard metric $|\cdot|$.  Equip  the Hilbert cube  $[0,1]^{\mathbb{N}}$ with the  product metric  
$$d(x,y)=\sum_{n\in \mathbb{N}}2^{-n}|x_n-y_n|,$$
and   let  $\sigma:[0,1]^{\mathbb{N}}\rightarrow [0,1]^{\mathbb{N}}$ be the left shift  given by $\sigma((x_n)_{n\in \mathbb{N}})=(x_{n+1})_{n\in \mathbb{N}}$.

$(1)$ It is  well known from \cite{lt18} that    $\overline{\rm mdim}_M(\sigma,[0,1]^{\mathbb{N}},d)=1$, and hence $$\overline{\rm mdim}_M^B(\sigma,[0,1]^{\mathbb{N}},d)=1$$ by Proposition \ref{prop 2.6}. 

Let $\epsilon >0$ and consider a $\sigma$-fixed point  $x_0=(0,0,...)$. Set $\alpha=\min\{\frac{1}{2}, \frac{\epsilon}{2}\}$. Then one can check that $[0,\alpha]^{\mathbb{N}}\subset W_\epsilon^s(x_0,\sigma)$.  Define $f_\alpha: [0,1]^{\mathbb{N}} \rightarrow [0,\alpha]^{\mathbb{N}}$ by $f_\alpha((x_n)_{n\in \mathbb{N}})=(\alpha x_n)_{n\in \mathbb{N}}$.  Since $d(f_\alpha(x),f_\alpha(y))=\alpha\, d(x,y)$ and $\sigma\circ f_\alpha = f_\alpha\circ \sigma$, the map $f_\alpha$ is a bi-Lipschitz conjugacy between the TDSs $([0,1]^\mathbb N,\sigma,d)$ and $([0,\alpha]^\mathbb N,\sigma,d)$.  Hence $\overline{\rm mdim}_M^B(\sigma,[0,\alpha]^{\mathbb{N}},d)=1$. By Proposition \ref{prop 2.6}, this  shows that
\begin{align*}
1\leq  \overline{\rm mdim}_M^B(\sigma,W_\epsilon^s(x_0,\sigma),d)&\leq  \overline{\rm mdim}_M^P(\sigma,W_\epsilon^s(x_0,\sigma),d)\\
&\leq \overline{\rm mdim}_M(\sigma,W_\epsilon^s(x_0,\sigma),d)\leq 1,
\end{align*}
and  for any $Q \in \{UC,B,P\}$,
\begin{align*}
\overline{\rm mdim}_M^Q(\sigma,W_\epsilon^s(x_0,\sigma),d)= \overline{\rm mdim}_M(\sigma,[0,1]^{\mathbb{N}},d)=1.
\end{align*}

$(2)$ Since  $x_0$ is a fixed point, we have $W_\epsilon^s(x_0,\sigma) \subset \sigma^{-n} W_\epsilon^s(x_0,\sigma)$ for all $n \in \mathbb{N}$. We get
\begin{align*}
1=\overline{\rm mdim}_M(\sigma,W_\epsilon^s(x_0,\sigma),d)&\leq \limsup_{\delta\to 0}\frac{1}{\log \frac{1}{\delta}} \limsup_{n\to \infty} \frac{1}{n}  \log r_n(\sigma, \sigma^{-n}W_\epsilon^s(x_0,\sigma),d,\delta)\\
&\leq \overline{\rm mdim}_M(\sigma,[0,1]^{\mathbb{N}},d)=1.
\end{align*}
This shows that $$\limsup_{\delta\to 0}\frac{1}{\log \frac{1}{\delta}} \limsup_{n\to \infty} \frac{1}{n}  \log r_n(\sigma, \sigma^{-n}W_\epsilon^s(x_0,\sigma),d,\delta)
= \overline{\rm mdim}_M(\sigma,[0,1]^{\mathbb{N}},d)=1.$$
\end{ex}

\subsection{Application to different entropy-like quantities in  infinite-entropy systems}

We have defined  metric mean dimension associated with topological entropy. From a topological viewpoint, however, other entropy‑like quantities, such as topological conditional entropy and preimage neighborhood entropy, have been introduced to describe dynamical complexity.

\emph{Question:  In infinite‑entropy systems, do the metric mean dimensions associated with these three entropy notions coincide?}


We  briefly recall the precise definitions of topological conditional entropy  and preimage neighborhood entropy.

\subsubsection{Topological conditional entropy}
To describe entropy that persists at arbitrarily small scales, Misiurewicz \cite{mis76} introduced \emph{topological conditional entropy} via open covers. Burguet \cite{bur09} later gave an equivalent  formulation using spanning sets.  We define the \emph{topological conditional entropy}  of $X$ as
\begin{align*}
h^{*}(T)&=\lim_{\epsilon \to 0}\lim\limits_{\delta \to 0}GR(\{\sup_{x\in X}r_n(T, B_n(x,\epsilon),d,\delta)\}_{n\geq 1}),\\
&=\inf_{\epsilon > 0}\lim\limits_{\delta \to 0}GR(\{\sup_{x\in X}r_n(T, B_n(x,\epsilon),d,\delta)\}_{n\geq 1}).
\end{align*}

\subsubsection{Preimage neighborhood entropy}

 For non-invertible dynamical systems,  the backward orbits may exhibit rather complicated  preimage structures. From a quantitative viewpoint, various preimage entropy-like quantities have been introduced to quantify the degree of non‑invertibility of a system and its contribution to entropy. Inspired by  topological preimage entropy \cite{cn05} for TDSs and stable topological entropy \cite{wz21}  in differentiable dynamical systems,  we define the \emph{preimage neighborhood entropy}  of $X$ as
\begin{align*}
	h_{pre}(T,X)&=
	\lim_{\epsilon \to 0}\lim_{\delta \to 0}GR(\{\sup_{x\in X, k \geq n} r_n(T, T^{-k}B(x,\epsilon),d,\delta)\}_{n\geq 1}),\\
	&=
	\inf_{\epsilon > 0}\lim_{\delta \to 0}GR(\{\sup_{x\in X, k \geq n} r_n(T, T^{-k}B(x,\epsilon),d,\delta)\}_{n\geq 1}).
\end{align*}

Observe that if the topological conditional entropy and the preimage neighborhood entropy are infinite, then, at every resolution $\epsilon>0$, both the entropy near a single orbit and the entropy of preimage neighborhoods of the phase space are infinite.
  Namely, for every $\epsilon >0$, one has
$$\lim\limits_{\delta \to 0}GR(\{\sup_{x\in X}r_n(T, B_n(x,\epsilon),d,\delta)\}_{n\geq 1})=\lim_{\delta \to 0}GR(\{\sup_{x\in X, k \geq n} r_n(T, T^{-k}B(x,\epsilon),d,\delta)\}_{n\geq 1})=\infty.$$ 

\begin{df}\label{df 5.1}
Let $(X,d,T)$ be a TDS. We define  the upper {conditional metric mean dimension of $X$  and  the  preimage neighborhood  metric mean dimension of $X$}   as
\begin{align*}
\overline{\rm mdim}_{M}^{*}(T,X,d,\epsilon)&=\limsup_{\delta \to 0}\frac{1}{\log\frac{1}{\delta} }GR(\{\sup_{x\in X}r_n(T, B_n(x,\epsilon),d,\delta)\}_{n\geq 1}),\\
\overline{\rm mdim}_{pre}(T,X,d,\epsilon)&=
\limsup_{\delta \to 0}\frac{1}{\log\frac{1}{\delta} }GR(\{\sup_{x\in X,k \geq n} r_n(T, T^{-k}B(x,\epsilon),d,\delta)\}_{n\geq 1}),
\end{align*}
respectively.
\end{df}
 The following theorem shows that the quantities in Definition \ref{df 5.1} are well‑defined and that the divergence rates of topological conditional entropy and preimage neighborhood entropy coincide with the corresponding rate for topological entropy.

\begin{thm}\label{thm 5.2}
Let $(X,d,T)$ be a TDS.  Then for every $\epsilon >0$,
$$\overline{\rm mdim}_{M}^{*}(T,X,d,\epsilon)=\overline{\rm mdim}_{pre}(T,X,d,\epsilon)=\overline{\rm mdim}_{M}(T,X,d).$$
\end{thm}

\begin{proof}
Fix $\epsilon>0$ and let $\delta >0$.  Notice that $$W_{\frac{\epsilon}{2}}^s(x,T)\subset \overline{B}_n(x,\frac{\epsilon}{2})\subset {B}_n(x,\epsilon) \subset X$$ for  all $n\in \mathbb{N}$ and  $x\in X$. Hence, one has 
$$\sup_{x\in X} r_n(T, W_{\frac{\epsilon}{2}}^s(x,T),d,\delta)\leq \sup_{x\in X} r_n(T, {B}_n(x,\epsilon),d,\delta) \leq r_n(T,X,d,\delta).$$
Applying  Theorem \ref{thm 1.2} with $f=0$,  we obtain, for every  $\epsilon>0$,
$$\limsup_{\delta \to 0}\frac{1}{\log\frac{1}{\delta} }GR(\{\sup_{x\in X}r_n(T, B_n(x,\epsilon),d,\delta)\}_{n\geq 1})=\overline{\rm mdim}_{M}(T,X,d).$$
This shows that $\overline{\rm mdim}_{M}^{*}(T,X,d,\epsilon)=\overline{\rm mdim}_{M}(T,X,d)$.

Similarly, using the fact  $T^{-n}W_{\frac{\epsilon}{2}}^s(x,T)\subset T^{-n}{B}(x, \epsilon)\subset X$
for every $n \in \mathbb{N}$ and $x\in X$,  one has 
$$\sup_{x\in X} r_n(T, T^{-n}W_{\frac{\epsilon}{2}}^s(x,T),d,\delta)\leq \sup_{x\in X, k\geq n} r_n(T, T^{-k}{B}(x, \epsilon),d,\delta) \leq r_n(T,X,d,\delta).$$
Applying  Theorem \ref{thm 1.3} with $f=0$,  we obtain, for every  $\epsilon>0$,  $$\limsup_{\delta \to 0}\frac{1}{\log\frac{1}{\delta} }GR(\{\sup_{x\in X,k \geq n} r_n(T, T^{-k}B(x,\epsilon),d,\delta)\}_{n\geq  1})=\overline{\rm mdim}_{M}(T,X,d).$$
This shows that $\overline{\rm mdim}_{pre}(T,X,d,\epsilon)=\overline{\rm mdim}_{M}(T,X,d)$.
\end{proof}

\begin{rem}
$(1)$ Although the aforementioned entropy-like quantities play different roles in characterizing the topological complexity of dynamical systems, the divergence  rates of their $\delta$-entropies  coincide with the metric mean dimension of $X$. Consequently, the two quantities  defined in Definition \ref{df 5.1} are independent of $\epsilon$. Therefore, we write
\[
\overline{\mathrm{mdim}}_{{pre}}(T,X,d):=\overline{\mathrm{mdim}}_{{pre}}(T,X,d,\epsilon)
\]
for the preimage neighborhood metric mean dimension of $X$, and
\[
\overline{\mathrm{mdim}}_{M}^{*}(T,X,d):=\overline{\mathrm{mdim}}_{M}^{*}(T,X,d,\epsilon)
\]
for the upper conditional metric mean dimension of $X$.		
	
$(2)$ In \cite{lr24}, the authors replaced  $T^{-k}B(x,\epsilon)$ by $T^{-k}\{x\}$, and defined  the preimage metric mean dimension of $X$, which we denote  by $\overline{\rm mdim}_{pre}^{*}(T,X,d)$. We  cannot  expect that  Theorem \ref{thm 5.2}  holds for this quantity, because  $\overline{\rm mdim}_{pre}^{*}(T,X,d)$ is zero for every invertible  TDS $(X,d,T)$ with positive  upper metric mean dimension.
\end{rem}

We give an example to clarify  Theorem \ref{thm 5.2}.

\begin{ex}	
In the setting of Example \ref{ex 3.7},   we have
$$\overline{\rm mdim}_{M}^{*}(\sigma,[0,1]^{\mathbb{N}},d)=\overline{\rm mdim}_{pre}(\sigma,[0,1]^{\mathbb{N}},d)=\overline{\rm mdim}_{M}(\sigma,[0,1]^{\mathbb{N}},d)=1.$$
\begin{proof}

It suffices to show $\overline{\rm mdim}_{M}^{*}(\sigma,[0,1]^{\mathbb{N}},d)\geq 1$ and  $\overline{\rm mdim}_{pre}(\sigma,[0,1]^{\mathbb{N}},d)\geq 1$, since we already know  $\overline{\rm mdim}_{M}(\sigma,[0,1]^{\mathbb{N}},d)=1$ \cite{lt18}.  

Fix $0<\epsilon<\frac{1}{2}$ and let  $\delta \in (0,1)$.  Choose  $x\in [0,1]^{\mathbb{N}}$ with $0\leq x_n\leq \frac{1}{2}$ for any $n \geq 0$. Notice that 
$$E:=\{y\in [0,1]^{\mathbb{N}}: y_j\in [x_j,x_j+\frac{\epsilon}{3}]~\text{for}~0\leq j\leq n-1~\text{and}~y_j=x_j~\text{for}~j\geq n\}\subset B_n(x,\epsilon).$$
Set $A^{i}=\{x_i+j\delta :j=0,...,[\frac{\epsilon}{3\delta}]\}$ for every $0\leq i\leq n-1$,
and  let 
$$Z_1=\{(y_0,..,y_{n-1},x_{n},x_{n+1},...):y_i\in A^{i}~ \text{for}~ i=0,...,n-1\} $$
 be a subset of $E$.
Then  $Z_1$  is an $(n,\delta)$-separated  set of $B_n(x,\epsilon)$  and hence
$$\sup_{x\in  [0,1]^{\mathbb{N}}}r_{n}(\sigma,B_n(x,\epsilon),d,\frac{\delta}{2})\geq \sup_{x\in  [0,1]^{\mathbb{N}}}s_{n}(\sigma,B_n(x,\epsilon),d,\delta)\geq  ([\frac{\epsilon}{3\delta}]+1)^{n} \geq (\frac{\epsilon}{3\delta})^n.$$
This shows that  $\overline{\rm mdim}_{M}^{*}(\sigma,[0,1]^{\mathbb{N}},d)\geq 1$. 

Notice that  for every $x \in [0,1]^{\mathbb{N}}$, one has 
$$\sigma ^{-n}(x)=\{(y_0,..,y_{n-1},x_{0},x_1,...):y_i\in [0,1]~ \text{for}~ i=0,...,n-1\}.$$
Set $B=\{j\delta :j=0,...,[\frac{1}{\delta}]\}$
and  let 
$$Z_2=\{(y_0,..,y_{n-1},x_{0},x_1,...):y_i\in B~ \text{for}~ i=0,...,n-1\}. $$
Then  $Z_2$  is an $(n,\delta)$-separated  set of $\sigma ^{-n}(x)$  and hence
$$\sup_{x\in  [0,1]^{\mathbb{N}}, k \geq n}r_n(\sigma,\sigma ^{-k}(x),d,\frac{\delta}{2})\geq\sup_{x\in  [0,1]^{\mathbb{N}}, k \geq n}s_n(\sigma,\sigma ^{-k}(x),d,\delta)\geq  ([\frac{1}{\delta}]+1)^{n} \geq (\frac{1}{\delta})^n.$$
We obtain that $\overline{\rm mdim}_{pre}(\sigma,[0,1]^{\mathbb{N}},d)\geq 1.$

To sum up, we get  $$\overline{\rm mdim}_{M}^{*}(\sigma,[0,1]^{\mathbb{N}},d)=\overline{\rm mdim}_{pre}(\sigma,[0,1]^{\mathbb{N}},d)=\overline{\rm mdim}_{M}(\sigma,[0,1]^{\mathbb{N}},d)=1.$$
\end{proof}
\end{ex}


\section{Metric mean dimension of stable sets}
In this section, we reveal chaotic phenomena in infinite‑entropy systems and characterize the ``size" of the  chaotic sets in terms of upper capacity metric mean dimension,
 proving Theorems \ref{thm 1.4} and \ref{thm 1.5}.

\subsection{Chaotic phenomena in  infinite-entropy systems}
For an invertible TDS $(X,T)$,  the  \emph{stable set} and \emph{unstable set} of $x$ for $T$ are respectively given by
\begin{align*}
W^s(x,T)&=\{y\in X:\lim_{n\to \infty}d(T^nx,T^ny)=0\},\\
W^u(x,T)&=\{y\in X:\lim_{n\to \infty}d(T^{-n}x,T^{-n}y)=0\}.
\end{align*}

It is immediate that  $W^s(x,T)=W^u(x,T^{-1})$ and  $W^u(x,T)=W^s(x,T^{-1})$ for all $x\in X$. We investigate what kinds of chaotic properties  arise in stable and unstable sets, and introduce the following definition.

\begin{df} \cite[Definition 4.1]{bh08}
Let $(X,T)$ be a TDS and $A\subset X$ be a non-empty closed subset. The set $A$ is said to be  weakly mixing  for $T$ if  there exists $B\subset A$ satisfying
\begin{enumerate}
\item $B$ is the union of countably many Cantor sets;
\item the closure of $B$ equals $A$;
\item  for every $C \subset B$ and every continuous map $g: C\rightarrow A$, there exists an increasing sequence of natural numbers  $\{n_i\}_{i\geq 1}$ such that $\lim\limits_{i \to \infty}T^{n_i}x=g(x)$ for every $x \in C$.
\end{enumerate}
\end{df}

It follows from the  above  definition that  weak mixing of a set is  a setwise  rather than a global property. When $\#X\geq 2$,   Xiong and Yang  \cite{xy91} proved that   $X$ is a weakly mixing  set if and only if  the system $(X,T)$ is weakly mixing (i.e., $(X\times X, T\times T)$ is transitive).  Thus, this setwise notion agrees with  the usual global notion when $A=X$. Denote by  $WM_s(X,T)$  the set of  all weakly mixing sets for $T$. The weakly mixing sets also admit the following hitting-time characterization \cite[Proposition 4.2]{bh08}.

\begin{prop}\label{prop 4.2}
Let $(X,T)$ be a TDS, and let $A$ be a non-singleton closed subset of $X$.  Then $A$ is a weakly mixing subset for $T$  if and only if for any $k\in\mathbb{N}$, any choice of non-empty open subsets $V_1,\dots,V_k$ of $A$ and open subsets $U_1,\dots,U_k$ of $X$ with $A\cap U_i\neq\varnothing
$, $i=1,2,\dots,k$, there exists $m\in\mathbb{N}$ such that $T^m V_i\cap U_i\neq\varnothing
$ for each $1\le i\le k$.
\end{prop}

The following result shows that, for an invertible dynamical system, there exists a set that is weakly mixing for both $T$ and $T^{-1}$. In particular, if $\overline{\rm MRID}_{\mu}(T,X,d)>0$, this set also has positive upper capacity metric mean dimension with respect to $T$.

\begin{thm} \label{thm 4.5}
	Let $(X,d,T)$ be an  invertible TDS,  and let $\mu$ be an ergodic   Borel probability measure with  $h_{\mu}(T)>0$. Then for $\mu$-a.e. $x \in X$,  there exists  a closed subset $E(x) \subset \overline{W^s(x,T)}\cap \overline{W^u(x,T)}$ such that
	\begin{enumerate}
		\item  $E(x)$ is weakly mixing for both  $T$ and $T^{-1}$;
		\item $\overline{\rm {mdim}}_M^{UC}(T,E(x),d)\geq \overline{\rm MRID}_{\mu}(T,X,d)$.
	\end{enumerate}
\end{thm}

\subsubsection{Some tools  in local entropy theory}

We briefly recall some basic concepts and  collect the  key    lemmas  needed to prove Theorem \ref{thm 4.5}.

Denote by $\mathcal{C}_X$, $\mathcal{P}_X$, and $\mathcal{C}_X^o$ the collections of finite Borel covers, finite Borel partitions, and finite open covers of $X$, respectively. Given  $\mathcal{U},\mathcal{V} \in \mathcal{C}_X$, we write $\mathcal{V}\succ\mathcal{U}$ to denote that  $\mathcal{V}$ is \emph{finer} than $\mathcal{U}$, that is, each element of $\mathcal{V}$ is contained in some element of $\mathcal{U}$. The \emph{join} of   $\UU$ and $\mathcal{V}$ is  the cover $\UU \vee \mathcal{V}:=\{U\cap V:  U\in \UU, V\in \mathcal{V}\}$.  Denote by $\diam (\mathcal{U})=\max_{U \in \mathcal{U}}\diam U$  the \emph{diameter} of $\mathcal{U}$, and by  $\Leb(\UU)$ the \emph{Lebesgue number} of $\UU$ (i.e.,  the largest positive number $\delta$ such that every open ball $B(x,\delta)$, $x\in X$, is contained in some element of $\UU$). Given  a non-empty set $Z\subset X$ and $\UU\in \mathcal{C}_X^o$, put 
$$N(\UU|Z)=\min\{\#\mathcal{F}: \mathcal{F}\subset \UU~\text{and}\cup_{F\in \mathcal{F}}F\supset Z\}$$
and 
$$h(T, \mathcal{U}|Z)=\limsup_{n \to \infty}\frac{1}{n}\log N(\UU^n|Z),$$
where $\UU^n:=\vee_{j=0}^{n-1}T^{-j}\UU$ is the join of the open covers $\{T^{-j}\UU\}_{j=0}^{n-1}$.

 We formulate an equivalent definition for the  upper capacity metric mean dimension of  subsets using open covers.  

\begin{prop}\label{lem 3.4}
Let $(X,d,T)$ be a TDS. Then for  any non-empty subset $Z\subset X$,
\begin{align*}
\overline{\rm {mdim}}_M^{UC}(T,Z,d)&=\limsup_{\epsilon\to 0}\frac{1}{\logf}\inf_{\diam (\UU) \leq \epsilon}h(T, \mathcal{U}|Z)\\
&=\limsup_{\epsilon\to 0}\frac{1}{\logf}\sup_{\diam (\UU) \leq \epsilon, \Leb(\UU)\geq \frac{\epsilon}{4}}h(T, \mathcal{U}|Z)\\
&=\limsup_{\epsilon\to 0}\frac{1}{\logf}\sup_{\Leb(\UU)\geq \frac{\epsilon}{4}}h(T, \mathcal{U}|Z).
\end{align*}

\end{prop}
\begin{proof}
Fix $\epsilon >0$. Let  $\UU \in \mathcal{C}_X^o$ with $\diam (\UU) \leq \frac{\epsilon}{2}$.  Then  $r_n(T,Z, d,\epsilon)\leq  N(\UU^n|Z) $ and hence   $h_{top}(T, Z,d,\epsilon) \leq  \inf_{\diam (\UU) \leq\frac{ \epsilon}{2}}h(T, \mathcal{U}|Z)$.  

Let $\UU$ be  a  finite open cover of $X$  with  $\diam (\UU)\leq \epsilon$ and  $\Leb(\UU)\geq \frac{\epsilon}{4}$\footnote[1]{Such an open cover always exists (cf. \cite[Lemma 3.4]{gs20}). For instance, let $\UU:=\{B_d(x,\frac{\epsilon}{2}):x\in E\}$ be a finite open cover of $X$, where $E$ is a finite $\frac{\epsilon}{4}$-net of $X$.}.   Let   $F $  be an $(n,\frac{\epsilon}{4})$-spanning set of $Z$.  Then for $x\in F$, the  Bowen ball $B_n(x,\frac{\epsilon}{4})$  must be   contained in some element of $\mathcal{U}^n$. So    $r_n(T,Z, d,\frac{\epsilon}{4})\geq N(\UU^n|Z) $. It follows that $$h_{top}(T, Z,d,\frac{\epsilon}{4}) \geq  \sup_{\Leb(\UU) \geq \frac{\epsilon}{4}} h(T, \mathcal{U}|Z) \geq  \sup_{\diam (\UU) \leq \epsilon, \Leb(\UU)\geq \frac{\epsilon}{4}}  h(T, \mathcal{U}|Z)\geq  \inf_{\diam (\UU) \leq \epsilon}h(T, \mathcal{U}|Z).$$
Then the above two inequalities imply the desired  equalities.
\end{proof}

Let $M(X), M(X,T),$ and $ E(X,T)$ denote the sets of Borel probability measures on $X$,  $T$-invariant Borel probability measures on $X$,  $T$-invariant ergodic Borel probability measures on $X$, respectively.  Let $ \mu \in M(X)$,  $\alpha \in \mathcal{P}_X$,  and  let $\mathcal{C}$ be a sub-$\sigma$-algebra of $ B_\mu$, where  $B_\mu$ is the completion of  the Borel $\sigma$-algebra $\mathcal{B}(X)$ with respect to $\mu$. We put 
$$H_{\mu}(\alpha |\mathcal{C})=\sum_{A\in \alpha} \int - \mathbb{E}(\chi_A|\mathcal{C})\log \mathbb{E}(\chi_A|\mathcal{C})d\mu,$$
where  $\mathbb{E}(\chi_A|\mathcal{C})$ is the conditional expectation of $\chi_A$  with respect to $\mathcal{C}$.

If $\UU \in \mathcal{C}_X$,  we let
$H_{\mu}(\UU |\mathcal{C})=\inf_{\alpha \succ \UU, \alpha \in \mathcal{P}_X}H_{\mu}(\alpha |\mathcal{C})$.  If  $\mathcal{C}$ is a $T$-invariant sub-$\sigma$-algebra (i.e., $T^{-1}\mathcal{C}=\mathcal{C}$) and $\mu \in M(X,T)$, then  the  limit 
$$h_{\mu}(T, \UU |\mathcal{C})=\lim\limits_{n\to \infty}\frac{1}{n}H_{\mu}(\UU^n |\mathcal{C})$$
exists.
In particular,  if  $\mathcal{C}=\{\varnothing, X\}(\text{mod }\mu)$, we write $H_{\mu}(\mathcal{U}):=H_{\mu}(\mathcal{U}|\mathcal{C})$ and  $h_{\mu}(T, \UU):=h_{\mu}(T, \UU |\mathcal{C})$. 
Then for any $\mu \in M(X,T)$, by \cite[Lemma 2.7]{h08}, the measure‑theoretic entropy of $T$ with respect to $\mu$ is given by $$h_{\mu}(T)=\sup_{\mathcal{U} \in \mathcal{C}_X^o}h_{\mu}(T, \UU).$$


Let $\mu \in  {M}(X)$, $\epsilon>0$, $n \in \mathbb{N}$ and $\delta \in (0,1)$.
Put
$$R_\mu^\delta(T,n, \epsilon):=\min\{\#E: E\subset X  ~\text{and} ~\mu (\cup_{x\in E}B_n(x,\epsilon))\geq 1-\delta \}.$$
We define the {upper  Katok $\epsilon$-entropy of $\mu$} as
\begin{align*}
&\overline{h}_{\mu}^K(T,\epsilon)=\lim_{\delta \to 0}\limsup_{n\to \infty} \frac{1}{n} \log R_\mu^\delta(T,n, \epsilon).
\end{align*}

We collect some standard facts to clarify the relations among these  measure-theoretic entropy-like quantities.
\begin{lem}\label{lem 4.4}
Let $(X,d,T)$ be an invertible TDS, $\mu \in M(X,T)$,  and let $\UU \in\mathcal{C}_X^o$. Then 
\begin{enumerate}
\item  If $\mu \in E(X,T)$,  then $ h_{\mu}(T, \UU)=h_{\mu}(T, \UU|P_{\mu}(T))$,  where  $P_{\mu}(T)$ is the Pinsker $\sigma$-algebra of $(X,B_\mu,\mu,T)$.
\item If $\mu \in E(X,T)$, we have
$$\overline{\rm MRID}_{\mu}(T,X,d)=\overline{\rm mdim}_{\mu}^K(T,X,d)=\limsup_{\epsilon\to 0}\frac{ 1}{\logf}\sup_{\diam (\UU) \leq \epsilon, \atop \Leb(\UU)\geq \frac{\epsilon}{4}}h_\mu(T,\mathcal{U}),$$
and $$\overline{\rm mdim}_M(T,X,d)=\limsup_{\epsilon \to 0} \frac{1}{\logf}\sup_{\mu \in E(X,T)} \sup_{\diam (\UU) \leq \epsilon, \atop \Leb(\UU)\geq \frac{\epsilon}{4}}h_\mu(T,\mathcal{U}).$$
\item  If $\mu=\int_X \mu_xd\mu(x)$  is the  disintegration of  $\mu$ over the sub-$\sigma$-algebra  $\mathcal{C}$ of $\mathcal{B}_{\mu}$, then 
$$H_{\mu}(\UU |\mathcal{C})=\int_X H_{\mu_x}(\UU) d\mu(x).$$

\end{enumerate}
\end{lem}
\begin{proof} 
(1) and (3)  follow from  \cite[Lemma 2.1]{h08} and \cite[Lemma 2.2]{h08}, respectively.

(2) follows from \cite[Lemma 16 (1)]{y25} (see also \cite[Theorems 1.1-1.3]{ycz25}).
\end{proof}

\subsubsection{Proof of  Theorem  \ref{thm 4.5}}

\begin{proof}
Let $B_\mu$ be the completion of the Borel $\sigma$-algebra $\mathcal{B}(X)$ with respect to $\mu$. Then $(X,B_\mu,\mu,T)$ is a Lebesgue system.
  Let
$\mu=\int_X \mu_xd\mu(x)$ be the  disintegration of $\mu$ over $P_{\mu}(T)$.    The arguments in Steps 1 and 2 of the proof of \cite[Theorem 4.6]{h08}, which remain valid whenever $h_\mu(T)>0$, imply that, for $\mu$-a.e.\ $x\in X$,
 $$\text{supp}(\mu_x)\subset \overline{W^s(x,T)}\cap \overline{W^u(x,T)} $$ and  $\operatorname{supp}(\mu_x)\in WM_s(X,T)  \cap WM_s(X,T^{-1})$.  By \cite[Corollary 5.24]{ew13},  one has $T_{*}\mu_x=\mu_{Tx}$ for $\mu$-a.e. $x\in X$  since the  Pinsker $\sigma$-algebra  $P_{\mu}(T)$ is $T$-invariant.   Let $X_0\subset X$ be  a $T$-invariant  set  of full $\mu$ measure  such that $T_{*}\mu_x=\mu_{Tx}$ for all $x\in X_0$. 

 For each $\UU \in \mathcal{C}_X^o$, we  verify that the function $$x\in X_0 \mapsto H_{\mu_x}(\UU)= \inf_{\alpha \succ \UU, \alpha \in \mathcal{P}_X}H_{\mu_x}(\alpha)$$ is Borel  measurable. Assume that $\UU=\{U_1,...,U_m\}$.   For every  $s=(s_1,...,s_m)\in \{0,1\}^m$, we define the set
$U_s=\bigcap_{j=1}^mU_j(s_j),$
where $U_j(0)=U_j$ and $U_j(1)=X\backslash U_j$ for every $1\leq j \leq m$.  Then $\alpha=\{U_s:s\in \{0,1\}^m,\ U_s\neq\varnothing\} \in \mathcal{P}_X$.   By \cite[Proposition 6]{rom03}, we have $$H_{\mu_x}(\UU)=\min_{\beta \in P(\mathcal{U})}\limits H_{\mu_x}(\beta),$$
where  $P(\UU)=\{\beta \in \mathcal{P}_X:\alpha \succ\beta\succ\UU\}$ is a finite family of partitions of $X$.  Fix   $\beta \in P(\mathcal{U})$. The  characterization of conditional measures \cite[Theorem 5.14]{ew13}   shows that for $x\in X_0$,
\begin{align}
H_{\mu_x}(\beta)&=\sum_{B\in\beta}-\mu_x(B)\log\mu_x(B)\\
&=\sum_{B\in\beta}- \mathbb{E}(\chi_B|P_{\mu}(T))(x)\log  \mathbb{E}(\chi_B|P_{\mu}(T))(x)\nonumber.
\end{align}
This yields  the measurability of $H_{\mu_x}(\UU)$   since  $\mathbb{E}(\chi_B|P_{\mu}(T))(x)$ is Borel measurable.  For each $n\ge 1$, define the measurable function $f_n$ on $X_0$ by $f_n(x)=H_{\mu_x}(\UU^n)$. For each $x\in X_0$,   one  has
$$f_{n+m}(x)\leq H_{\mu_x}(\UU^n)+H_{\mu_x}(T^{-n}\UU^m)=f_{n}(x)+f_{m}(T^nx).$$  Then  Kingman's subadditive ergodic theorem ensures that $\lim_{n\to \infty}\frac{1}{n}f_n(x)$  exists for  $\mu$-a.e. $x\in X$; denote this limit by $a_{\UU}$. By Lemma \ref{lem 4.4} (1) and (3),  we get 
\begin{align}
h_{\mu}(T, \UU)
=h_{\mu}(T, \UU|P_{\mu}(T))
&=\lim_{n\to \infty}\frac{1}{n}H_{\mu}(\UU^n|P_{\mu}(T))\nonumber\\
&=\lim_{n\to \infty}\int \frac{1}{n}H_{\mu_x}(\UU^n)d\mu\nonumber\\
&=\int  \lim_{n\to \infty}\frac{1}{n}H_{\mu_x}(\UU^n)d\mu=a_{\UU}\nonumber.
\end{align}
Since   
$f_n(x)=H_{\mu_x}(\UU^n)\leq \log N(\UU^n|\text{supp}(\mu_x))$ for all $x\in X_0$ and $n\geq 1$, this implies that 
\begin{align}\label{inequ 3.14}
a_{\UU}=h_{\mu}(T, \UU) \leq h(T,\UU|\text{supp}(\mu_x))
\end{align}
for  $\mu$-a.e. $x\in X$. By Lemma \ref{lem 4.4} (2), choose a sequence $\epsilon_k \to 0$ as $k \to \infty$,  and    a family of  finite open covers $\{\UU_{k}\}_{k\geq 1}$ of $X$  with $\diam (\UU_k) \leq \epsilon_k$ and $ \Leb(\UU_k)\geq \frac{\epsilon_k}{4}$,  such  that $\overline{\rm MRID}_{\mu}(T,X,d)= \lim_{k \to \infty}\frac{1}{\log\frac{1}{\epsilon_k}} h_{\mu}(T,\UU_k).$
By (\ref{inequ 3.14}),  choose a  full $\mu$-measure set $X_1$ such that for  each  $x\in X_1$, 
\begin{align}\label{inequ 4.6}
h_{\mu}(T, \UU_k) \leq h(T,\UU_k|\text{supp}(\mu_x))
\end{align}
for every $k\geq 1$.  Therefore,  by (\ref{inequ 4.6}), for every $x\in X_1$, we have 
\begin{align*}
\overline{\rm MRID}_{\mu}(T,X,d)
	&\leq \limsup_{k\to \infty}\frac{1}{\log\frac{1}{\epsilon_k}}\sup_{\diam (\UU) \leq \epsilon_k, \Leb(\UU)\geq \frac{\epsilon_k}{4}}h(T,\UU|\text{supp}(\mu_x))\\
	&\leq \overline{\rm mdim}_M^{UC}(T,\text{supp}(\mu_x),d)~ \text{by Proposition \ref{lem 3.4}}.
\end{align*}
\end{proof}

\subsubsection{Metric mean dimension and  natural extension} 
We use the natural extension theorem of metric mean dimension  to derive   the non-invertible  version of Theorem \ref{thm 4.5}.

For a  surjective non-invertible dynamical system, the metric mean dimension of its natural extension  has been considered in \cite [Lemma 3.8]{bs24} with respect to a certain product metric. Using a different product metric, we establish corresponding natural extension results for upper capacity metric mean dimension and  upper mean Katok entropy.

Let $(X,d,T)$  be  a   non-invertible TDS. We define the set
$$\widetilde{X}=\{(x_n)_{n\in \mathbb{Z}}:x_{n+1}=Tx_n, \forall n \in \mathbb{Z}\},$$
and endow it
with  the  subspace topology of the product  topology on $\prod_{j\in \mathbb{Z}}X$  induced by the metric:
$$\widetilde{d}(\widetilde{x},\widetilde{y})=\sum_{n\in \mathbb{Z}}\frac{d(x_n,y_n)}{2^{|n|}},$$
where $\widetilde{x}= (x_n)_{n\in \mathbb{Z}}$ and $\widetilde{y}=(y_n)_{n\in \mathbb{Z}}$. Let $\widetilde{T}$ be  the (left) shift map  on $\widetilde{X}$ given by  $$\widetilde{T}((x_n)_{n\in\mathbb{Z}})=(x_{n+1})_{n\in\mathbb{Z}}.$$

When $T$ is surjective, $\pi: \widetilde{X}\rightarrow X$  is  a  factor map, and $\widetilde{X}$ can be identified with the space of all bi-infinite  orbits  in $X$. The invertible dynamical system  $(\widetilde{X},\widetilde{T})$ is called the \emph{natural extension} of $(X,T)$. 



\begin{lem}\label{lem 4.8}
Let  $(X,d,T)$ be a  TDS,  and let $(\widetilde{X},\widetilde{T})$  be its natural extension. Let $\pi: \widetilde{X} \rightarrow X$ be the projection to the $0$-th coordinate.
Then for any non-empty  subset  $E$ of $\widetilde{X}$ and  $\mu \in M(\widetilde{X})$,
\begin{align*}
{\rm {\overline{mdim}}}_M^{UC}(\widetilde{T},E,\widetilde{d})&= {\rm {\overline{mdim}}}_M^{UC}(T,\pi(E),d),\\
\limsup_{\epsilon \to 0}\frac{\overline{h}_{\mu}^K(\widetilde{T},\epsilon)}{\logf}&=\limsup_{\epsilon \to 0}\frac{\overline{h}_{\pi_{*}\mu}^K(T,\epsilon)}{\logf}.
\end{align*} 

Consequently, if $T$ is surjective, then ${\rm {\overline{mdim}}}_M^{UC}(\widetilde{T},\widetilde{X},\widetilde{d})= {\rm {\overline{mdim}}}_M^{UC}(T,X,d)$. 
\end{lem}

\begin{proof}
$(1)$. We first  show  ${\rm {\overline{mdim}}}_M^{UC}(\widetilde{T},E,\widetilde{d})= {\rm {\overline{mdim}}}_M^{UC}(T,\pi(E),d).$

For every  $\widetilde{x} \in \widetilde{X}$,  we  denote  the image $\pi(\widetilde{x})$ of $\widetilde{x}$ under $\pi$ by $x_0$. The fact  $\widetilde{d}(\widetilde{x},\widetilde{y})\geq d(x_0,y_0)$  for
any $\widetilde{x}, \widetilde{y} \in \widetilde{X}$ implies that $$\widetilde{d}_n(\widetilde{x},\widetilde{y})\geq d_n(x_0,y_0)$$ for all $n\geq 1$. Then for every $\epsilon>0$, if $F$ is an  $(n,\epsilon)$-spanning set of $E$, then $\pi(F)$ is an $(n,\epsilon)$-spanning set of $\pi(E)$.  This  means that
\begin{align}\label{equ 4.9}
r_n(T,\pi(E),d,\epsilon)\leq r_n(\widetilde{T},E,\widetilde{d},\epsilon),
\end{align}
and hence ${\rm {\overline{mdim}}}_M^{UC}(\widetilde{T},E,\widetilde{d})\geq  {\rm {\overline{mdim}}}_M^{UC}(T,\pi(E),d)$.

Fix $\epsilon >0$ and choose $N_0$ (depending on $\epsilon$) such that  $\sum_{|j|>N_0} \frac{\diam(X,d)}{2^{|j|}}<\epsilon$, and let  $\UU$  be a finite open cover of $X$ with $\diam(\UU)\leq \epsilon$.  For every word  $\underline{U}=(U_j)_{j\in J_{n,N_0}}$, where $U_{j} \in \mathcal{U}$ and   $J_{n,N_0}:=[-N_0,-1]\cup[n,n+N_0-1]$, we  define the corresponding cylinder set along the two‑boundary blocks by
$$\widetilde{X}_n(\underline{U}):=\{\widetilde{x} \in \widetilde{X}:x_j\in U_j\text{ for every }j\in J_{n,N_0}\}.$$
If $x\in X$ and  $\widetilde{X}_n(\underline{U}) \cap \pi^{-1}B_n(x,\epsilon)\not=\varnothing
$, then for some
$\widetilde{y}_{x,\underline{U}} \in  \widetilde{X}_n(\underline{U}) \cap \pi^{-1}B_n(x,\epsilon)$, we have
$$\widetilde{X}_n(\underline{U}) \cap \pi^{-1}B_n(x,\epsilon)\subset B_n(\widetilde{y}_{x,\underline{U}},7\epsilon).$$
Indeed, for any $\widetilde{z}\in \widetilde{X}_n(\underline{U}) \cap \pi^{-1}B_n(x,\epsilon)$,  we have  $d(z_j, (\widetilde{y}_{x,\underline{U}})_j)<2\epsilon$ for all $j\in [-N_0,n+N_0-1]$. Therefore,   for every $0\leq k \leq n-1$, one has
\begin{align*}
\widetilde{d}({\widetilde{T}}^{k}\widetilde{z}, {\widetilde{T}}^{k} \widetilde{y}_{x,\underline{U}})\leq \sum_{j=-N_0}^{N_0}\frac{d(({\widetilde{T}}^{k}\widetilde{z})_j, ({\widetilde{T}}^{k} \widetilde{y}_{x,\underline{U}})_j)}{2^{|j|}}+\epsilon
< 3\cdot 2\epsilon +\epsilon=7\epsilon.
\end{align*}
This yields  that
\begin{align}\label{inequ 4.5}
\pi^{-1}B_n(x,\epsilon)\subset \bigcup_{  \widetilde{X}_n(\underline{U}) \cap \pi^{-1}B_n(x,\epsilon)\not=\varnothing
}B_n(\widetilde{y}_{x,\underline{U}},7\epsilon).
\end{align}
 Now let $F\subset X$ be an  $(n,\epsilon)$-spanning set of $\pi(E)$. For each $x\in F$, let $F_x$ be the set of points $ \widetilde{y}_{x,\underline{U}}$ appearing in (\ref{inequ 4.5}). Then $\#F_x \le (\# \mathcal{U})^{2N_0}$.  Since
\begin{align*}
E\subset\cup_{x\in F}\pi^{-1}B_n(x,\epsilon)\subset \cup_{x\in F}\cup_{\widetilde{y}\in F_x}B_n(\widetilde{y},7\epsilon),
\end{align*} 
we have
\begin{align}\label{equ 4.8}
r_n(\widetilde{T},E,\widetilde{d},7\epsilon)\leq r_n({T},\pi(E),d,\epsilon)\cdot (\#\mathcal{U})^{2N_0}.
\end{align}
Observe that  $\#\UU$  depends only  on $\epsilon$, not on $n$. We have $h_{top}^{UC}(\widetilde{T}, E,\widetilde{d},7\epsilon) \leq h_{top}^{UC}(T, \pi(E),d,\epsilon)$, and hence  ${\rm {\overline{mdim}}}_M^{UC}(\widetilde{T},E,\widetilde{d})\leq {\rm {\overline{mdim}}}_M^{UC}(T,\pi(E),d)$ is verified.

$(2)$. Next we  show  $\limsup_{\epsilon \to 0}\frac{\overline{h}_{\mu}^K(\widetilde{T},\epsilon)}{\logf}=\limsup_{\epsilon \to 0}\frac{\overline{h}_{\pi_{*}\mu}^K(T,\epsilon)}{\logf}.$

The  fact $\widetilde{d}_n(\widetilde{x},\widetilde{y})\geq d_n(x_0,y_0)$ means that
$B_n(\widetilde{x},\epsilon)  \subset \pi^{-1}(B_n(x_0,\epsilon))$ for every  $\widetilde{x} \in \widetilde{X}$. Let $\mu \in M(\widetilde{X})$, $\delta \in (0,1)$ and assume that $\mu(\bigcup_{j=1}^k  B_n(\widetilde{x_j},\epsilon) ) \geq 1-\delta$, where  $ \widetilde{x_j} \in \widetilde{X}$ for every $1\leq j \leq k$. Then 
$$\pi_{*}\mu(\bigcup_{j=1}^k  B_n(\pi(\widetilde{x_j}),\epsilon))\geq 1-\delta.$$
This implies that   $R_{\pi_{*}\mu}^\delta({T},n, \epsilon) \leq R_{\mu}^\delta(\widetilde{T},n, \epsilon)$, and hence $\overline{h}_{\pi_{*}\mu}^K(T,\epsilon) \leq \overline{h}_{\mu}^K(\widetilde{T},\epsilon)$. We get
$$\limsup_{\epsilon \to 0}\frac{\overline{h}_{\pi_{*}\mu}^K(T,\epsilon)}{\logf} \leq \limsup_{\epsilon \to 0}\frac{\overline{h}_{\mu}^K(\widetilde{T},\epsilon)}{\logf}.$$

From (\ref{inequ 4.5}), one obtains  $ R_{\mu}^\delta(\widetilde{T},n, 7\epsilon) \leq R_{\pi_{*}\mu}^\delta(T,n, \epsilon) \cdot (\#\mathcal{U})^{2N_0}$. This inequality  yields the reverse inequality
$$\limsup_{\epsilon \to 0}\frac{\overline{h}_{\pi_{*}\mu}^K(T,\epsilon)}{\logf} \geq \limsup_{\epsilon \to 0}\frac{\overline{h}_{\mu}^K(\widetilde{T},\epsilon)}{\logf}.\qedhere$$
\end{proof} 


\subsubsection{Proofs of Theorems \ref{thm 1.4} and \ref{thm 1.5}}

\begin{proof}[Proof of Theorem \ref{thm 1.4}]
Since $\overline{\rm mdim}_{\mu}^K(T,X,d)>0$,  we have $h_{\mu}(T)=\infty$.		
Let $(\widetilde{X},\widetilde{T})$  be  the  natural extension of $(X,T)$, where  $\pi: \widetilde{X}  \rightarrow  X $ is  the projection onto the $0$-th coordinate. Since $\pi_{*}:E(\widetilde X,\widetilde T)\to E(X,T)$ is surjective, we choose $\nu \in E(\widetilde{X},\widetilde{T})$  satisfying $\pi_{*}\nu=\mu$ (cf. \cite[Proposition 3.11]{dgs76}). By Theorem \ref{thm 4.5},  there exists a  Borel set $\widetilde{X}_0\subset  \widetilde{X}$ of  full $\nu$-measure such that, for every  $\widetilde{x} \in \widetilde{X}_0$,  there exists a  closed weakly mixing subset $E(\widetilde{x}) \subset \overline{W^s(\widetilde{x},\widetilde{T})}$ such that  
$$\overline{\rm {mdim}}_M^{UC}(\widetilde{T},E(\widetilde{x}),\widetilde{d})\geq \overline{\rm MRID}_{\nu}(\widetilde{T},\widetilde{X},\widetilde{d})=\overline{\rm mdim}_{\nu}^K(\widetilde{T},\widetilde{X},\widetilde{d}),$$ 
where the equality follows from  Lemma \ref{lem 4.4} (2).
For every $n\geq 1$, choose  a compact set $K_n\subset \widetilde{X}_0 $ satisfying $\nu(K_n)>1-\frac{1}{n}$. We put $$X_0:=\bigcup_{n=1}^{\infty}\pi(K_n).$$
Each $\pi(K_n)$ is compact, hence $X_0$ is Borel; furthermore, for every $n$,
\[
\mu(X_0)\ge \mu(\pi(K_n))=\nu\big(\pi^{-1}(\pi(K_n))\big)\ge \nu(K_n)>1-\frac1n.
\]
 Hence $\mu(X_0)=1$.  For each  $x\in X_0$,  there exists $\widetilde{x} \in \widetilde{X}_0$ such that $\pi(\widetilde{x})=x$. Let  $E(x):=\pi(E(\widetilde{x})) $. Then $E(x) \subset \overline{W^s({x},{T})} $, and the fact that  $(X,d)$  is a compact metric space  implies that  $E(x)$   is closed. 

 By~Lemma~\ref{lem 4.8},  we obtain 
\begin{align*}
	\overline{\rm {mdim}}_M^{UC}({T},E(x),{d})
	= 
	\overline{\rm {mdim}}_M^{UC}(\widetilde{T},E(\widetilde{x}),\widetilde{d})
	\geq \overline{\rm mdim}_{\mu}^K({T},{X},{d})>0.
\end{align*}
This also shows that $E(x)$ is  not a singleton. Fix $k\ge 1$ and open sets $U_1,\dots,U_k,V_1,\dots,V_k\subset X$ such that $U_i\cap E(x)\neq\varnothing
$ and $V_i\cap E(x)\neq \varnothing
 $ for every $1\leq i \leq k$.
Since $U_i\cap E(x)\neq\varnothing
$ and $V_i\cap E(x)\neq\varnothing
$, we have
 $\pi^{-1}(U_i)\cap E(\widetilde{x})\neq\varnothing
 $ and $\pi^{-1}(V_i)\cap E(\widetilde{x})\neq\varnothing
 .$
Applying Proposition \ref{prop 4.2} to the relatively open subsets $\pi^{-1}(U_i)$  and  $\pi^{-1}(V_i)$ of $E(\tilde x)$, there exists  $n\ge 1$ such that for every $i=1,\dots,k$,
\[
\widetilde{T}^{n}\bigl(\pi^{-1}(U_i)\cap E(\widetilde{x})\bigr)\cap \pi^{-1}(V_i)\neq\varnothing
.
\]
Hence one can choose $\widetilde{y}_i\in \pi^{-1}(U_i)\cap E(\widetilde{x})$  such that  $\widetilde{T}^{n}\widetilde{y}_i\in \pi^{-1}(V_i).$
Let $y_i:=\pi(\widetilde{y}_i)$. Then
$y_i\in U_i\cap E(x),$ and 
\[
T^{n}y_i=T^{n}\pi(\widetilde{y}_i)=\pi(\widetilde{T}^{n}\widetilde{y}_i)\in V_i.
\]
It follows from Proposition \ref{prop 4.2} that $E(x)$ is a  weakly mixing set for $T$.
\end{proof}

\begin{prop}\label{prop 4.7}
Let  $(X,d,T)$ be a TDS. Then 
$$\overline{\rm {mdim}}_M(T,X,d)=\overline{\rm {mdim}}_M(T|_{X_{\infty}},X_{\infty},d),$$
where $X_{\infty}:=\bigcap_{j\geq 0}T^jX$.
\end{prop}

\begin{proof}
Using the fact $T^jX \subset T^kX$ for any $j>k$, we know that $X_{\infty}$ is a non-empty compact subset of $X$, and the restriction of $T$ to $ X_{\infty}$, i.e., $T|_{X_{\infty}}: X_{\infty}\rightarrow  X_{\infty}$, is a continuous surjective map. It is clear that $$\overline{\rm {mdim}}_M(T,X,d)\geq \overline{\rm {mdim}}_M(T|_{X_{\infty}},X_{\infty},d),$$ by the definitions.

Next we show the reverse inequality. Fix $\epsilon > 0$ and $m \in \mathbb{N}$. Set $q := r_m(T,X_\infty,d,\epsilon)$ and choose $z_1,\dots,z_q \in X$ with
 $$X_\infty \subset \bigcup_{i=1}^q B_m(z_i,\epsilon).$$
Choose $k \in \mathbb{N}$ such that $T^k X \subset\bigcup_{i=1}^q B_m(z_i,\epsilon)$. Let $p := r_k(T,X,d,\epsilon)$ and choose $x_1,\dots,x_p \in X$ such that
$X = \bigcup_{i=1}^p B_k(x_i,\epsilon). $ For $n > k$, put $\ell:=\left\lceil \frac{n-k}{m} \right\rceil$ and  $ \mathcal{I}_n:=\{1,\dots,p\}\times \{1,\dots,q\}^{\ell+1}.$ For every $\alpha=(a,i_0,\dots,i_{\ell})\in\mathcal{I}_n$, define
\[
V_\alpha := B_k(x_a,\epsilon) \cap \bigcap_{j=0}^{\ell} T^{-(k+jm)}B_m(z_{i_j},\epsilon).
\]
For each $\alpha=(a,i_0,\dots,i_{\ell})\in\mathcal{I}_n$ with $V_\alpha \not=\varnothing$, choose $y_\alpha \in V_\alpha$. Then $V_\alpha \subset B_n(y_\alpha,2\epsilon)$. Moreover,
$$X=\bigcup_{\alpha\in\mathcal{I}_n} V_\alpha.$$
Indeed, for every $x\in X$ and $0\le j\le\ell$, we have $$T^{k+jm}x\in T^{k+jm}X \subset T^k X  \subset \bigcup_{i=1}^q B_m(z_i,\epsilon)$$
and $x\in B_k(x_a,\epsilon)$ for some $1\leq a\leq p$. This shows that for any $n>k$,
$$r_n(T,X,d,2\epsilon) \le p\cdot q^{\ell+1}= r_k(T,X,d,\epsilon) \cdot r_m(T,X_\infty,d,\epsilon)^{\left\lceil \frac{n-k}{m} \right\rceil+1},$$
and hence
\begin{align*}
\frac1n \log r_n(T,X,d,2\epsilon) &\le \frac{\log r_k(T,X,d,\epsilon)}{n} + \frac{\left\lceil (n-k)/m \right\rceil +1}{n}\log r_m(T,X_\infty,d,\epsilon)\\
&\le \frac{\log r_k(T,X,d,\epsilon)}{n} + \frac{\left\lceil (n-k)/m \right\rceil +1}{n}\log r_m(T|_{X_\infty},X_\infty,d,\epsilon).
\end{align*}

Taking the upper limit as $n\to\infty$ yields
\[
h_{{top}}^{{UC}}(T,X,d,2\epsilon)\le \frac1m\log r_m(T|_{X_\infty},X_\infty,d,\epsilon)\qquad \forall m\in\mathbb{N}. 
\]
Consequently,
\[
\begin{aligned}
	h_{{top}}^{{UC}}(T,X,d,2\epsilon)
	\le \limsup_{m\to\infty}\frac1m\log r_m(T|_{X_\infty},X_\infty,d,\varepsilon)
	=h_{{top}}^{{UC}}(T|_{X_{\infty}},X_\infty,d,\epsilon).
\end{aligned}
\]
This implies that
$\overline{\rm {mdim}}_M(T,X,d)\leq \overline{\rm {mdim}}_M(T|_{X_{\infty}},X_{\infty},d).$
\end{proof}

\begin{proof}[Proof of Theorem \ref{thm 1.5}]
It is clear that $$\limsup_{\delta \to 0}\frac{1}{\log\frac{1}{\delta}}\sup_{x\in X}h_{top}(T,W^s(x,T),d,\delta) \leq \overline{\rm {mdim}}_M(T,X,d).$$
For the reverse inequality, by Proposition \ref{prop 4.7} it suffices to   prove 
$$\overline{\rm {mdim}}_M(T,X,d) \leq \limsup_{\delta \to 0}\frac{1}{\log\frac{1}{\delta}}\sup_{x\in X}h_{top}(T,W^s(x,T),d,\delta)$$
for the surjective TDS $(X,d,T)$.

Without loss of generality, we may assume that $\overline{\rm mdim}_M({T},{X},{d})>0$; otherwise, there is nothing to prove.
Let $(\widetilde{X},\widetilde{T})$  be  the  natural extension of $(X,T)$, where $\pi: \widetilde{X}\rightarrow X$ is the projection onto the $0$-th coordinate. 
Using Lemma \ref{lem 4.4} (2) and Lemma \ref{lem 4.8}, we  choose  a sequence $\delta_n \to 0$ as $n\to \infty$,  a sequence $\{\mu_n\}_{n\geq 1}$ of ergodic measures on  $\widetilde{X}$ and  a sequence $\{\mathcal{U}_n\}_{n\geq 1}$ of finite open covers of $\widetilde{X}$ with $\diam (\UU_n) \leq \delta_n$ and $ \Leb(\UU_n)\geq \frac{\delta_n}{4}$ such that $$\overline{\rm mdim}_M({T},{X},{d})=\overline{\rm mdim}_M(\widetilde{T},\widetilde{X},\widetilde{d})=\lim_{n \to \infty}\frac{1}{\log\frac{1}{\delta_n}}h_{\mu_n}(\widetilde{T},\mathcal{U}_n)>0.$$
Then for sufficiently large $n$,   $$h_{\mu_n}(\widetilde{T})=\sup_{\UU} h_{\mu_n}(\widetilde{T},\mathcal{U})\geq h_{\mu_n}(\widetilde{T},\mathcal{U}_n) >0.$$ For  each such $n$,  the proof of Theorem \ref{thm 4.5} shows that,  for  $\mu_n$-a.e. $\widetilde{x}\in \widetilde{X}$,  there exists a  closed set  $\text{supp}({(\mu_{n})}_{\widetilde{x}}) \subset \overline{W^s(\widetilde{x},{\widetilde{T}})} $ satisfying
\begin{align}\label{equ 3.14}
h_{\mu_n}(\widetilde{T}, \UU_n) \leq h(\widetilde{T},\UU_n|\text{supp}({(\mu_{n})}_{\widetilde{x}})) \leq \sup_{\widetilde{x}\in \widetilde{X}}h(\widetilde{T},\UU_n|\overline{W^s(\widetilde{x},{\widetilde{T}})}).
\end{align}
This implies that
\begin{align*}
\overline{\rm mdim}_M({T},{X},{d})&
\leq \limsup_{n \to \infty}\frac{1}{\log\frac{1}{\delta_n}}\sup_{\widetilde{x}\in \widetilde{X}}h(\widetilde{T},\UU_n|\overline{W^s(\widetilde{x},{\widetilde{T}})})\\
&\leq \limsup_{n \to \infty}\frac{1}{\log\frac{1}{\delta_n}}\sup_{\widetilde{x}\in \widetilde{X}}h_{top}(\widetilde{T},\overline{W^s(\widetilde{x},\widetilde{T})},\widetilde{d},\frac{\delta_n}{4})\\
&\leq \limsup_{\delta \to 0}\frac{1}{\log\frac{1}{\delta}}\sup_{\widetilde{x}\in \widetilde{X}}h_{top}(\widetilde{T},\overline{W^s(\widetilde{x},\widetilde{T})},\widetilde{d},\delta).
\end{align*}
Using the inequality $r_n(T,Z, d,\delta)\leq  r_n(T,\overline{Z}, d,\delta) \leq r_n(T,Z, d, \frac{\delta}{2}),$ we have
$$\limsup_{\delta \to 0}\frac{1}{\log\frac{1}{\delta}}\sup_{\widetilde{x}\in \widetilde{X}}h_{top}(\widetilde{T},\overline{W^s(\widetilde{x},\widetilde{T})},\widetilde{d},\delta)\\
=\limsup_{\delta \to 0}\frac{1}{\log\frac{1}{\delta}}\sup_{\widetilde{x}\in \widetilde{X}}h_{top}(\widetilde{T},W^s(\widetilde{x},\widetilde{T}),\widetilde{d},\delta).$$
Together with (\ref{equ 4.8}),  we obtain that
\begin{align*}
	&\limsup_{\delta \to 0}\frac{1}{\log\frac{1}{\delta}}\sup_{\widetilde{x}\in \widetilde{X}}h_{top}(\widetilde{T},\overline{W^s(\widetilde{x},\widetilde{T})},\widetilde{d},\delta)\\
\leq &\limsup_{\delta \to 0}\frac{1}{\log\frac{1}{\delta}}\sup_{\widetilde{x}\in \widetilde{X}}h_{top}(T,\pi (W^s(\widetilde{x},\widetilde{T})),d,\frac{\delta}{7})\\
	\leq &\limsup_{\delta \to 0}\frac{1}{\log\frac{1}{\delta}}\sup_{x\in {X}}h_{top}(T,W^s(x,{T}),d,\frac{\delta}{7}) \\
	\leq &\overline{\rm mdim}_M({T},{X},{d}).
\end{align*}

	
\end{proof}

\section{Final comments and further progress}

We conclude  this paper with several questions motivated by our main results.

Until now, we have obtained only partial results on the metric mean dimensions of $\epsilon$-stable sets and on chaotic phenomena arising in systems with infinite entropy.  Further work in the following two directions may help us better understand unpredictability and topological dynamics of infinite-entropy systems.

Comparing Theorems \ref{thm 1.2}, \ref{thm 1.3}, and  \ref{thm 1.5} with Theorem \ref{thm 1.1} leads us to ask whether the supremum $\sup_{x\in X}$ and the limit superior $\limsup_{\delta\to 0}$ can be exchanged. Although the answer is positive for certain dynamical systems, it is unclear whether there are counterexamples for which this interchange fails.
 Therefore, even in the zero-potential case we  pose the following   questions\footnote[3]{We  introduce the above open questions raised in an earlier version.   The answers are now known to be negative \cite[Theorem 2.6 and Proposition 5.8]{hzz26}.}: 

\begin{enumerate}
	\item [(1)]
Does
	\begin{align*}
		\overline{\rm mdim}_M(T,X,d)&=\sup_{x\in X} 	\overline{\rm mdim}_M^B(T,W_{\epsilon}^s(x,T),d)\\
		&=\sup_{x\in X} \overline{\rm mdim}_M^P(T,W_{\epsilon}^s(x,T),d)\\
		&=\sup_{x\in X} \overline{\rm mdim}_M^{UC}(T,W_{\epsilon}^s(x,T),d)
	\end{align*} 
hold for any  TDS $(X,d,T)$  and    $\epsilon >0$?  
	
	\item [(2)]  Does
	\begin{align*}
		\overline{\rm mdim}_M(T,X,d)
		=\sup_{x\in X}\limsup_{\delta\to 0}\frac{1}{\log \frac{1}{\delta}} \limsup_{n\to \infty} \frac{1}{n}  \log r_n(T, T^{-n}W_\epsilon^s(x,T),d,\delta)
	\end{align*} 
	hold for any  TDS $(X,d,T)$  and    $\epsilon >0$?  
	
	\item [(3)]  Does 
	
	\begin{align*}
		\overline{\rm mdim}_M(T,X,d)&=\sup_{x\in X} 	\overline{\rm mdim}_M^B(T,W^s(x,T),d)\\
		&=\sup_{x\in X} \overline{\rm mdim}_M^P(T,W^s(x,T),d)\\
		&=\sup_{x\in X} \overline{\rm mdim}_M^{UC}(T,W^s(x,T),d)
	\end{align*} 
	hold for any  TDS $(X,d,T)$?  
\end{enumerate}

Using the Bowen factor formula for Bowen topological entropy \cite{fhyz12} relating the factor and extension systems together with zero‑dimensional principal extensions \cite{bdo4}, Fang et al. \cite{fhyz12} estimated the Bowen topological entropies of stable and scrambled sets associated with chaotic phenomena in systems with finite positive entropy\footnote[4]{Indeed,   using the zero-dimensional principal extension  theorem of \cite{dh13}, their results also apply to   every  infinite-entropy system.}.
   The present work  also  stimulates us to  use  Bowen and packing metric mean dimensions to characterize the ``size" of   such chaotic sets. If we follow the techniques developed in \cite{fhyz12},  a significant challenge is how  to obtain  an   extension theorem for metric mean dimension, because the metric mean dimension depends on  the choice of a compatible metric on the phase space. In general, there are examples showing that the analogue of Bowen's factor formula  fails for Bowen metric mean dimension and that zero‑dimensional principal extensions need not preserve metric mean dimension between a factor system and its extension.  
It would be interesting to overcome this difficulty by developing new tools.

\section{Appendix}

\subsection{A note on Bowen metric mean dimension}

In this subsection, we elaborate on Remark \ref {rem 2.9} (3). More precisely, the two definitions of Bowen metric mean dimension on subsets, defined by Bowen's original approach and Pesin's approach \cite{b73,pp84}, coincide.

We first recall Bowen's approach  to   Bowen topological entropy  \cite{b73}.  Let $\mathcal{U}$ be a finite  open cover of $X$, and  let $Z$ be a non-empty subset   of $X$. For $E\subset X$, write  $E \prec \mathcal{U}$ if   $E$ is   contained in some element of $\mathcal{U}$. In particular, for a   family  $\{E_i\}_{i=1}^{\infty}$ of non-empty subsets of $X$, we write $\{E_i\}_{i=1}^{\infty} \prec \mathcal{U}$   if $E_i \prec \mathcal{U}$ for every $i\geq 1$.  We define 
$n_{T,\mathcal{U}}(E)$ as  the largest integer $m$ such that
$$T^{j}E\prec \mathcal{U}~\forall j=0,...,m-1,$$
and   put $n_{T,\mathcal{U}}(E)=0$ if $E \nprec \mathcal{U}$ and  $n_{T,\mathcal{U}}(E)=+\infty$ if $T^{j}E\prec \mathcal{U}$  for all $j\geq 0$. Then  the  ``diameter" of $E$ in the dynamical sense  is given by $D_{\mathcal{U}}(E):=e^{-n_{T,\mathcal{U}}(E)}$. For $s\geq 0$, we set
\begin{align*}
M_{\mathcal{U},s}(Z)&=\lim_{\epsilon \to 0}\inf\{\sum_{i}e^{-n_{T,\mathcal{U}}(E_i)s}:\cup_i E_i \supset Z, D_{\mathcal{U}}(E_i)<\epsilon\},\\
&=\lim_{N \to \infty}\inf\{\sum_{i}e^{-n_{T,\mathcal{U}}(E_i)s}:\cup_i E_i \supset Z, D_{\mathcal{U}}(E_i)\leq e^{-N}\}.
\end{align*}

There exists  a critical value  of $s$    at which $ M_{\mathcal{U},s}(Z)$ jumps  from  $\infty$ to $0$; it is defined by 
\begin{align*}
h_{\mathcal{U}}(T,Z):
=\inf\{s>0: M_{\mathcal{U},s}(Z)=0\}
=\sup\{s>0: M_{\mathcal{U},s}(Z)=\infty\}.
\end{align*}

The Bowen topological entropy of $Z$ is defined by
\begin{align*}
h_{top}^B(T,Z)=\sup_{\UU \in \mathcal{C}_X^o}h_{\mathcal{U}}(T,Z).
\end{align*}

We proceed to present Pesin's approach  to Bowen topological entropy. Let $\mathcal{W}_n(\mathcal{U})$ be the  collection of all strings  $\underline{U}:=U_{i_0}U_{i_1}\cdots U_{i_{n-1}}$ of length $n$, where $U_{i_j} \in \mathcal{U}$ for every $0\leq j \leq n-1$.  We define
$$X(\underline{U}):=\{x\in X: T^jx\in U_{i_j}~~ \forall j=0,...,n-1\},$$
and let $\mathcal{W}(\mathcal{U})=\cup_{n \geq 1}\mathcal{W}_n(\mathcal{U})$. A subcollection $\Gamma \subset \mathcal{W}(\mathcal{U})$  covers $Z$ if $Z \subset \cup_{\underline{U}\in \Gamma} X(\underline{U})$. Denote by $m(\underline{U})$  the length of   $\underline{U}$. Let $Z\subset X$ be a non-empty subset, $s\geq 0$ and $N\in \mathbb{N}$.
Define 
$$M_{N,s}(T,Z,\mathcal{U})=\inf_{\Gamma \subset \mathcal{W}(\mathcal{U})}\{\sum_{\underline U\in \Gamma}e^{-m(\underline{U})s}: \Gamma ~\text{covers}~ Z~~\text{and}~m(\underline{U})\geq N ~\text{for every }~\underline{U} \in \Gamma\}.$$
Let $M_{s}(T,Z,\mathcal{U})=\lim_{N\to \infty}M_{N,s}(T,Z,\mathcal{U})$.
It is readily checked that  there is  a critical value  $s$  such that $M_{s}(T,Z,\mathcal{U})$  jumps  from  $\infty$ to $0$; it is defined by 
\begin{align*}
\hat h_{\mathcal{U}}(T,Z):
=\inf\{s>0: M_{s}(T,Z,\mathcal{U})=0\}
=\sup\{s\geq 0: M_{s}(T,Z,\mathcal{U})=\infty\}.
\end{align*}

The Bowen topological entropy of $Z$ is defined by
\begin{align*}
	\hat{h}_{top}^B(T,Z)=\sup_{\UU \in \mathcal{C}_X^o}\hat{h}_{\mathcal{U}}(T,Z).
\end{align*}

\begin{prop}\label{prop 7.1}
For every  non-empty subset $Z$ of $X$, we have
\begin{align*}
\overline{\rm {mdim}}_M^B(T,Z,d)&=\limsup_{\epsilon \to 0} \frac{1}{
\logf} \inf_{\diam \mathcal{U} \leq \epsilon}\hat h_{\mathcal{U}}(T,Z)\\
&=\limsup_{\epsilon \to 0} \frac{1}{
\logf}\inf_{\diam \mathcal{U} \leq \epsilon}  h_{\mathcal{U}}(T,Z).
\end{align*}
\end{prop}

\begin{proof}
For the first equality, the left-hand side  is  the  Bowen metric mean dimension  defined by Bowen balls. By  considering a finite open cover  $\mathcal{U}$ of $X$ with  $\diam (\UU)\leq \epsilon$ and  $\Leb(\UU)\geq \frac{\epsilon}{4}$, the standard comparison between balls and cover names yields the desired equality. The second equality holds since $\hat h_{\mathcal{U}}(T,Z)=h_{\mathcal{U}}(T,Z)$ for every finite open cover $\mathcal{U}$ of $X$ (cf. \cite[Proposition 4, p. 317]{pp84}).
\end{proof}

\subsection{Further discussion on Theorem \ref{thm 1.3}}
Using  an argument  in the proof of the natural extension theorem for metric mean dimension presented in  Lemma \ref{lem 4.8} and Remark \ref{rem 3.5}, in the setting of zero potential we can improve  Theorem \ref{thm 1.3} by moving $\sup_{x\in X}$ before  $\limsup_{n \to \infty}$. 

\begin{thm}
	Let $(X,d,T)$ be a TDS.  Then  for every $\epsilon >0 $, 
	\begin{align*}
		\overline{\rm mdim}_M(T,X,d)=\limsup_{\delta\to 0}\frac{1}{\log \frac{1}{\delta}} \sup_{x\in X} \limsup_{n\to \infty} \frac{1}{n}  \log r_n(T, T^{-n}W_\epsilon^s(x,T),d,\delta).
	\end{align*}   
\end{thm}

\begin{proof}
It suffices to show
$$\overline{\rm mdim}_M(T,X,d) \leq	\limsup_{\delta\to 0}\frac{1}{\log \frac{1}{\delta}} \sup_{x\in X} \limsup_{n\to \infty} \frac{1}{n}  \log r_n(T, T^{-n}W_\epsilon^s(x,T),d,\delta) .$$
By  Proposition \ref{prop 4.7}, we need only prove it when $T:X\rightarrow X$ is surjective.
Let $(\widetilde{X},\widetilde{T})$  be  the  natural extension of $(X,T)$, where  $\pi: \widetilde{X}  \rightarrow  X $ is  the projection onto the $0$-th coordinate.

Fix $\widetilde{x}\in \widetilde{X}$, $n\in \mathbb{N}$, $\epsilon >0$ and  let $x=\pi(\widetilde{x})$. For any  $\widetilde{y} \in  \widetilde{T}^{-n}W_\epsilon^{s,*}(\widetilde{x},\widetilde{T})$ and $j\geq 0$, we have
	\[
	\begin{aligned}
		d(T^{n+j}\pi\widetilde{y},T^j x) &= d(\pi \widetilde{T}^{n+j}\widetilde{y},\pi \widetilde{T}^j\widetilde{x})\\
		&\le \widetilde{d}\big(\widetilde{T}^j(\widetilde{T}^n\widetilde{y}),\widetilde{T}^j\widetilde{x}\big)\le \epsilon.
	\end{aligned}
	\]
	Thus $T^n\pi\widetilde{y}\in W_\epsilon^s(x,T)$, and hence $\pi(\widetilde{T}^{-n}W_\epsilon^{s,*}(\widetilde{x},\widetilde{T})) \subset T^{-n}(W_\epsilon^s(x,T))$.  Let $S=\widetilde{T}^{-1}$. The map
\[
\widetilde{T}^{n-1}: \widetilde{T}^{-n}W_\epsilon^{s,*}(\widetilde{x},\widetilde{T})  \longrightarrow \widetilde{T}^{-1}W_\epsilon^{s,*}(\widetilde{x},\widetilde{T}) = W_\epsilon^{s,*}(S\widetilde{x},S)
\]
is a bijection, where the equality is
 verified as follows: 
\[
\begin{aligned}
	\widetilde{z}\in \widetilde{T}^{-1}W_\epsilon^{s,*}(\widetilde{x},\widetilde{T})
	&\iff \widetilde{d}(\widetilde{T}^{k+1}\widetilde{z},\widetilde{T}^k\widetilde{x})\le \epsilon \quad\text{for every }k\in\mathbb Z\\
	&\iff \widetilde{d}(S^j\widetilde{z},S^j S\widetilde{x})\le \epsilon \quad\text{for every }j\in\mathbb Z\\
	&\iff \widetilde{z}\in W_\epsilon^{s,*}(S\widetilde{x},S).
\end{aligned}
\]
Besides, for $\widetilde{u},\widetilde{v}\in \widetilde{T}^{-n}W_\epsilon^{s,*}(\widetilde{x},\widetilde{T})$, we have
\[
\begin{aligned}
	\widetilde{d}_n(\widetilde{u},\widetilde{v})
	&= \max_{0\le i<n}\widetilde{d}(\widetilde{T}^i\widetilde{u},\widetilde{T}^i\widetilde{v})\\
	&= \max_{0\le j<n}\widetilde{d}\big(S^j\widetilde{T}^{n-1}\widetilde{u},S^j\widetilde{T}^{n-1}\widetilde{v}\big).
\end{aligned}
\]
This implies that for every $\delta >0$,
\begin{align}\label{equ 6.1}
r_n(\widetilde{T},\widetilde{T}^{-n}W_\epsilon^{s,*}(\widetilde{x},\widetilde{T}),\widetilde{d},7\delta)=r_n\big(S,W_\epsilon^{s,*}(S\widetilde{x},S),\widetilde{d},7\delta\big). 	
\end{align}
By  (\ref{equ 4.8}) and (\ref{equ 6.1}),  there exists $C(\delta)$  (only depending on $\delta$) such that
\[
\begin{aligned}
	r_n\big(T,T^{-n}W_\epsilon^s(x,T),d,\delta\big)
	&\ge r_n(T,\pi (\widetilde{T}^{-n}W_\epsilon^{s,*}(\widetilde{x},\widetilde{T})),d,\delta)\\
	&\ge C(\delta)^{-1}r_n\big(S,W_\epsilon^{s,*}(S\widetilde{x},S),\widetilde{d},7\delta\big).
\end{aligned}
\]
Therefore,
\[
h_{\text{top}}^{{UC}}\big(S,W_\epsilon^{s,*}(S\widetilde{x},S),\widetilde{d},7\delta\big)\le {GR}\left(\big\{r_n\big(T,T^{-n}W_\epsilon^s(x,T),d,\delta\big)\big\}_{n\ge 1}\right).
\]
Taking the supremum over $\widetilde{x}$ and noting that  $S$ is onto,  using Remark  \ref{rem 3.5} we conclude that
\begin{align*}
{\rm {\overline{mdim}}}_M(T,X,d)=&{\rm {\overline{mdim}}}_M(\widetilde{T},\widetilde{X},\widetilde{d})\\
\leq&  \limsup_{\delta\to 0}\frac{1}{\log \frac{1}{\delta}}\sup_{x\in X}{GR}\left(\big\{r_n\big(T,T^{-n}W_\epsilon^s(x,T),d,\delta\big)\big\}_{n\ge 1}\right). 
\end{align*}
The reverse inequality is obvious. This completes the proof.
\end{proof}


\section*{Acknowledgement} 

This  manuscript was  substantially improved during the Workshop on ``\emph{Topological Dynamical Systems and Ergodic Theory}" held at Shantou University in January 2022. We would like to thank Prof. Jian Li, Dr. Xianjuan Liang, and the Dynamical Systems research group for their warm hospitality and the excellent academic environment they provided. Special thanks also go to Wen Huang, Bingbing Liang, Lei Jin, and Tao Wang for their valuable suggestions and helpful comments on an earlier version of this manuscript. We are also grateful to the  referees for many insightful comments and constructive suggestions that greatly improved the quality of this paper.

The first author was  supported by   NSFC (No. 12671229) and Xi'an International Science and Technology Cooperation Base: Ergodic Theory and Dynamical Systems.  
The  second author was supported by the
National Natural Science Foundation of China (Nos. 12471184 and 12071222).
The  third author   was  supported by the
National Natural Science Foundation of China (No. 11971236) and Qinglan Project of Jiangsu Province. 



\end{document}